\documentclass[final]{siamart190516}

\usepackage{amsfonts}
\usepackage{graphicx}
\usepackage{epstopdf}
\usepackage{algorithmic}
\usepackage{bbm}
\usepackage{enumerate}
\ifpdf
  \DeclareGraphicsExtensions{.eps,.pdf,.png,.jpg}
\else
  \DeclareGraphicsExtensions{.eps}
\fi

\newsiamremark{remark}{Remark}
\newsiamremark{hypothesis}{Hypothesis}
\crefname{hypothesis}{Hypothesis}{Hypotheses}
\newsiamthm{claim}{Claim}
\newsiamthm{assumption}{Assumption}
\headers{Existence of MPE}{M. Pitera and {\L}. Stettner}
\title{Existence of bounded solutions to multiplicative Poisson equations under mixing property\\
\medskip
{\scriptsize (Previous title: ''A note on Multiplicative Poisson Equation: developments in the span-contraction approach'')}}

 \author{
Marcin Pitera\thanks{Institute of Mathematics, Jagiellonian University, Krakow, Poland
  (\email{marcin.pitera@uj.edu.pl});  research supported by NCN grant 2020/37/B/ST1/00463.}
  \and
{\L}ukasz Stettner\thanks{Institute of Mathematics, Polish Academy of Sciences, Warsaw, Poland
  (\email{l.stettner@impan.pl});  research supported by NCN grant 2020/37/B/ST1/00463.}}

\usepackage{amsopn}

\usepackage{enumitem}

\newcommand{\1}{\mathbbm{1}}

\def\bP{\mathbb{P}}
\def\bE{\mathbb{E}}
\def\bR{\mathbb{R}}

\def\bN{\mathbb{N}}

\def\bC{\mathbb{C}}

\def\cP{\mathcal{P}}
\def\cE{\mathcal{E}} 

\makeatletter
\def\namedlabel#1#2{\begingroup
    #2%
    \def\@currentlabel{#2}%
    \phantomsection\label{#1}\endgroup
}
\makeatother

\newtheorem{example}{Example}[section]

\begin{document}

\maketitle

% REQUIRED
\begin{abstract}
In this paper we study the problem of Multiplicative Poisson Equation (MPE) bounded solution existence in the generic discrete-time setting. Assuming mixing and boundedness of the risk-reward function, we investigate what conditions should be imposed on the underlying non-controlled probability kernel or the reward function in order for the MPE bounded solution to always exists. In particular, we consolidate span-norm framework based results and derive an explicit sharp bound that needs to be imposed on the cost function to guarantee the bounded solution existence under mixing. Also, we study the properties which the probability kernel must satisfy to ensure existence of bounded MPE for any generic risk-reward function and characterise process behaviour in the complement of the invariant measure support. Finally, we present numerous examples and stochastic-dominance based arguments that help to better understand the intricacies that emerge when the ergodic risk-neutral mean operator is replaced with ergodic risk-sensitive entropy.
\end{abstract}

% REQUIRED

 %REQUIRED

\section{Introduction}
The main goal of this paper is to discuss the conditions which imply existence of a {\it Multiplicative Poisson Equation} (MPE) bounded solution in a generic discrete-time setup with pre-imposed mixing, see \cite{BalMey2000,DiMSte2006,CavHer2009,CavHer2010} for the general context. The problem considered in this paper is related to the entropic long-run risk sensitive averaged (per unit of time) performance assessment and its ergodic properties, see \cite{BieCiaPit2013,ste2022,KonMey2003}. Namely, we consider an uncontrolled Markov chain framework with bounded risk-reward setup. Given a risk-reward function $g$ and Markov chain $(x_i)_{i\in\bN}$ with starting point $x=x_0$, we are interested in conditions which imply ergodic stability and existence of the time-averaged value
\begin{equation}\label{eq:dlgi}
\lim_{n\to\infty}\tfrac{\mu_x\left(\sum_{i=0}^{n-1}g(x_i)\right)}{n},
\end{equation}
where $\mu_x(\cdot)=\ln \bE_x[\exp(\cdot)]$ is the normalised entropic utility under measure $\bP_x$ linked to the starting point, see Section~\ref{S:introduction} for the exact setup. 

In the wider context, the objective value considered in \eqref{eq:dlgi} is linked to the {\it risk sensitive criterion} which is a key object in the risk sensitive stochastic control framework. In fact, MPE bounded solution existence and strong mixing statements are often a prerequisite in controlled environments, when the corresponding Bellman equations are analysed, see~\cite{DiMSte2006,BisPra2022,Cav2018,BisBor2022} and references therein. In particular, solutions  to MPE are crucial in the study of unbounded horizon risk sensitive stopping problems which in turn are the main tool to solve risk sensitive impulse control problems, see \cite{JelPitSte2020,PitSte2019}. The usefulness of MPE analysis could be linked to the fact that a bounded MPE solution allow us to change probability measure and obtain much simpler stopping impulse control problem, see e.g. \cite{JelPitSte2021}.

Using the span-contraction framework, we analyse conditions that should be imposed on the underlying probability kernel or the reward function in order for the MPE bounded solution to always exists. In particular, we recall the local contraction property and derive an explicit sharp bound that needs to be imposed on the risk-reward function to guarantee the bounded solution existence under mixing. This results could be easily extended to the controlled framework and complement multiple conditions stated in the literature which e.g. state that for sufficiently small risk-aversion the solution could be found or provide a non-sharp bound conditions, see e.g.~\cite{Ste1999, PitSte2016,KonMey2003} and references therein. Also, we study the properties which the probability kernel must satisfy to ensure existence of bounded MPE for any generic risk-reward function and characterise process behaviour in the complement of the invariant measure support. Our analysis allows to better understand the classification made in~\cite{CavHer2009} and its potential expansion to a generic state spaces and non-norm-like cost-reward functions, see \cite{Cav2018,Cav2009}. We also present numerous examples that help to better understand the intricacies that emerge when the risk-neutral mean operator is replaced with risk-sensitive entropy in the ergodic setup.

This paper is organised as follows: In Section~\ref{S:introduction} we state the general setup and formulate all assumptions core to our study. Next, in Section~\ref{S:local.contr} we recall the central result on which the span-contraction framework is built on, i.e. the local contraction property of the MPE operator. We also provide an extensive comment on the subtleties which might result in local contraction not becoming the global contraction.
In Section~\ref{S:predReward} we present one of the main results of this paper linked to a bound that should be imposed on the risk-reward function for the bounded MPE solution to always exists under mixing.
Then, in Section~\ref{MPE.always} we study the problem of existence of bounded MPE solution for generic reward function, characterise how the process should behave in the complement of the invariant set, and show that MPE solution is in fact typically expected to be unbounded. Finally, in Section~\ref{S:negative} we study how the geometric propagation is linked to MPE solution existence and provide a series of negative results which show many intricacies inherently build into the ergodic risk-sensitive problem, even in the uncontrolled case.

\section{Preliminaries}\label{S:introduction}
Let $(E, \cE)$ be a locally compact separable metric space endowed with metric $\rho$ and Borel $\sigma$-field ${\cal{E}}$ that corresponds to the {\it state space}. In particular, this covers denumerable and finite case, where, for simplicity, we often assume that $E=\bN$ or $E=\{1,2,\ldots,N\}$ for some $N\in\bN$. %, with natural $\sigma$-field $2^E$.
We use ${\cal{P}}(E)$ to denote the space of probability measures on $(E,\cE)$ and $\bC(E)$ to denote the space of probability transition kernels $\bP$, such that $\mathbb{P}(x,\cdot)\in {\cal{P}}(E)$, for $x\in E$, and $\bP$ satisfy the Feller property, i.e. the mapping
\begin{equation}
x \mapsto \mathbb{P} f(x):=\int_E f(y) \mathbb{P}(x,dy)
\end{equation}
is continuous for every $f\in C(E)$, where $C(E)$ denotes the space of continuous and bounded functions on $E$. For now, let us fix $\bP\in\bC(E)$. Then, for a fixed starting point $x\in E$ and $f\in C(E)$, the normalised entropic utility is given by
\[
\mu_{x}(f):=\ln\left(\int_{E}e^{f(y)}\mathbb{P}(x,dy)\right).
\]
For a fixed {\it reward function} $g\in C(E)$ we consider the {\it Multiplicative Poisson Equation} (MPE) given by
\begin{align}\label{eq:3:bellman.av}
w(x)=g(x)-\lambda+\mu_x(w),\quad x\in E,
\end{align}
for $w\in C(E)$ and $\lambda \in \bR$. If it exists, we call $(\lambda,w)$ a solution to \eqref{eq:3:bellman.av}. The associated MPE operator is given by
\begin{equation}\label{lcont}
Tf(x):=g(x) +\mu_x(f);
\end{equation}
note that \eqref{eq:3:bellman.av} could be simply restated as $Tw(x)=w(x)+\lambda$.

For simplicity, we consider the normalised MPE with positive risk-aversion parameter equal to $\gamma=1$, see~\cite{BiePli1999} for details. That saying, all results presented in this paper directly transfer to a generic risk-sensitive MPE with entropy given by $\mu_x^\gamma(f)=\frac{1}{\gamma}\mu_x(\gamma f)$, for $\gamma\neq 0$. Indeed, by rescaling values of $g$, $\lambda$ and $w$, we can normalise the problem for any $\gamma\neq 0$.

For completeness, let us recall that given a (bounded) MPE solution \eqref{eq:3:bellman.av} we can immediately recover the the long-run average reward constant. With slight abuse of notation, given an arbitrary random variable $X$, $x\in E$ and $\bP\in\bC(E)$, we use $\mu_x(X):=\ln\bE_{x}[e^{X}]$ to denote normalised entropic utility of $X$. Also, we use $(x_i)_{i\in\bN}$ to denote the corresponding Markov process with starting point $x_0=x$ and refer to
\[
\textstyle J_x((x_i)_{i\in\bN}):=\liminf_{n\to\infty}\frac{\mu_x\left(\sum_{i=0}^{n-1}g(x_i)\right)}{n}
\]
as time-averaged normalised long-run entropy for $(x_i)_{i\in\bN}$, see~\cite{BieCiaPit2013} and~\cite{BiePli2003} for economic context and connections to the risk sensitive criterion performance measure. Next proposition shows that if MPE solution \eqref{eq:3:bellman.av} exists, then we immediately get $\lambda=J_x((x_i)_{i\in\bN})$. In particular, note that this value does not depend on the starting point $x\in E$.

\begin{proposition}\label{pr:lambda.is.good} Let $\bP\in \bC(E)$ and $g\in C(E)$ be such that a bounded MPE solution $(w,\lambda)\in C(E)\times \bR$ exists. Then
\begin{equation}\label{e1}
\textstyle \sup_{x\in E}\left|\lambda - \frac{1}{n}  \mu_x\left(\sum_{i=0}^{n-1} g(x_i)\right)\right|\to 0,\quad n\to\infty.
\end{equation}
\end{proposition}
\begin{proof}
Iterating MPE equation \eqref{eq:3:bellman.av}, for any $x\in E$ and $n\in\bN$, we get
\[
\textstyle w(x)=\mu_x\left(\sum_{i=0}^{n-1}(g(x_i)-\lambda)+w(x_n)\right)  
\]
which implies, due to monotonicity and translation invariance of the entropic utility,
\begin{equation}\label{eq:lambda.is.good}
\textstyle \sup_{x\in E}\left|\lambda - \frac{1}{n}\mu_x\left(\sum_{i=0}^{n-1} g(x_i)\right)\right|\leq 2\frac{\|w\|}{n}.
\end{equation}
Taking the limit of \eqref{eq:lambda.is.good} we complete the proof of \eqref{e1}.
\end{proof}

The main goal of this paper is to investigate under what conditions the solution to \eqref{eq:3:bellman.av} exists. Namely, we pre-assume a mixing condition and assess what additional conditions needs to be imposed either on $g\in C(E)$ or $\bP\in \bC(E)$ to guarantee existence of solution to \eqref{eq:3:bellman.av}. We say that the transition kernel $\bP\in \bC(E)$ satisfies the {\it mixing condition} with $\Lambda\in (0,1)$ if
\begin{equation}\tag{A.1}\label{A.1}
\sup_{x,x'\in E}  \left|\mathbb{P}(x,B)-\mathbb{P}(x',B)\right|\leq \Lambda,\quad B\in\cE.
 \end{equation}
For brevity, with slight abuse of notation, we often refer to $\Lambda$ as the minimal constant for which \eqref{A.1} is satisfied. Assumption \eqref{A.1} could be relaxed by considering multi-step dynamics. Namely, we say that the transition kernel satisfies the {\it multi-step mixing condition} with $\Lambda\in (0,1)$ and $n\in\bN$ if
\begin{equation}\tag{A.1'}\label{A.1'}
 \sup_{x,x'\in E}  \left|\mathbb{P}_n(x,B)-\mathbb{P}_n(x',B)\right|\leq \Lambda,\quad B\in\cE,
 \end{equation}
where $\bP_n(x,\cdot)$ denotes the $n$-step iterated transition kernel. In the literature, conditions \eqref{A.1} and \eqref{A.1'} are sometimes replaced by a slightly stronger condition linked to {\it global minorization}. For completeness, we also state these conditions and sometimes provide direct alternative proofs. We say that the transition kernel $\bP\in \bC(E)$ satisfies the {\it global minorization} condition for $d>0$ if there exists a probability measure $\eta\in \cP(E)$, such that
\begin{equation}\tag{A.2}\label{A.2}
\inf_{x\in E}\mathbb{P}(x,B) \geq d\eta(B),\quad B\in \cal{E}.
\end{equation}
Similarly, we say the transition kernel satisfies the {\it multi-step global minorization} condition for $d>0$ and $n\in\bN$, if there exists a probability measure $\eta$, such that
\begin{equation}\tag{A.2'}\label{A.2'}
\inf_{x\in E}\mathbb{P}_n(x,B) \geq d\eta(B),\quad B\in \cal{E}.
\end{equation}
For more information about  conditions \eqref{A.1} and \eqref{A.2} we refer to \cite{DiMSte1999}. Essentially, \eqref{A.1} refers to the total variation distance while \eqref{A.2} states the existence of the minorisation measure which in turn is almost equivalent to the Doeblin condition; see Theorem 16.2.1 and Theorem 16.2.3 in \cite{MeynTweedie} or Section 7.3 in \cite{HerLas1999}. It should be noted that any of the conditions \eqref{A.1}, \eqref{A.1'}, \eqref{A.2} or \eqref{A.2'} imply the existence of the (unique) invariant measure $\nu\in\cP(E)$, see e.g. Remark 7.3.13 in \cite{HerLas1999}. For brevity, if not stated otherwise, given the transition kernel $\bP\in\bC(E)$ we always use $\nu$ to denote its unique invariant measure (assuming it exists). For completeness, we also present some simple relations for the stated conditions in Proposition~\ref{pr:assumption.relation}

 \begin{proposition}\label{pr:assumption.relation}
Let $\bP\in\bC(E)$. Then

\begin{enumerate}
    \item If \eqref{A.2'} holds for $d\in (0,1)$, then \eqref{A.1'} holds for $\Lambda=1-d$, for the same $n\in\bN$.
    \item  If \eqref{A.1} holds for $\Lambda\in (0,1)$, then \eqref{A.1'} holds for any $n\in\bN$ with $\Lambda^n$.
    \item If \eqref{A.1'} holds and the unique invariant measure $\nu$ has an atom then \eqref{A.2'} holds for some $n\in\bN$.
\end{enumerate}
 \end{proposition}

 \begin{proof}
 (1.) For simplicity we only show the proof for \eqref{A.1} and \eqref{A.2}. Let us assume $\bP\in \bC(E)$ satisfy \eqref{A.2} for $d\in (0,1)$ and $\eta\in \cP(E)$. Then, for any $B\in \cE$, we get
 \begin{align*}
\sup_{x,x'\in E}  \left|\mathbb{P}(x,B)-\mathbb{P}(x',B)\right| &= \sup_{x\in E} \mathbb{P}(x,B)-\inf_{x\in E} \mathbb{P}(x,B)\\
& =1-\left(\inf_{x\in E} \mathbb{P}(x,B^c)+\inf_{x\in E} \mathbb{P}(x,B)\right)\\
& \leq 1-d.
\end{align*}
(2.) Let us assume that $\bP\in \bC(E)$ satisfy \eqref{A.1} for $\Lambda\in (0,1)$. We know that \eqref{A.1'} is satisfied for $n=1$. Let us now show that if $\eqref{A.1'}$ is satisfied for $n\in\bN$ and $\Lambda^{n}$ then $\eqref{A.1'}$ is satisfied for $n+1$ and $\Lambda^{n+1}$. Let us fix $x,x'\in E$ and let $H\in \cE$ denote the positive set from the  Hahn-Jordan decomposition applied to the signed measure $\nu_{x,x'}:=(\bP_n(x,\cdot)-\bP_n(x',\cdot))$. Then, for any $B\in\cE$, we get
\begin{align*}
\bP_{n+1}(x,B)-\bP_{n+1}(x',B) &= \int_{E}\bP(z,B)\nu_{x,x'}(dz)\\
&\leq\int_{H}\sup_{z\in E}\bP(z,B)\nu_{x,x'}(dz)+\int_{H^c}\inf_{z\in E}\bP(z,B)\nu_{x,x'}(dz)\\
& =\left[\sup_{z\in E}\bP(z,B)-\inf_{z\in E}\bP(z,B)\right]\int_{H}\nu_{x,x'}(dz)\\
& \leq \Lambda |\bP_n(x,H)-\bP_n(x',H)|\\
& \leq \Lambda^{n+1}
\end{align*}
Similarly, we get $\bP_{n+1}(x,B)-\bP_{n+1}(x',B)\geq -\Lambda^{n+1}$, which concludes the proof.

\noindent (3.) For simplicity, we only show the proof for \eqref{A.1}. From \cite{HerLas1999}, we know that for any $B\in\cE$ we have
\[
\sup_{x\in E}|\bP_n(x,B)-\nu(B)|\to 0, \quad n\to\infty.
\]
Let us assume there exists $y\in E$ such that $\nu(\{y\})>0$. Then, for sufficiently big $n\in \bN$, we get  $\sup_{x\in E}|\bP_n(x,\{y\})-\nu(\{y\}|)< \mu(\{y\})/2$ and consequently
\begin{align*}
\inf_{x\in E}\bP_n(x,\{y\}) &= \inf_{x\in E}(\bP_n(x,\{y\})-\nu(\{y\})) +\nu(\{y\})\\
&\geq \nu(\{y\})-\sup_{x\in E}|\bP_n(x,\{y\})-\nu(\{y\})|\\
& \geq \nu(\{y\})/2.
\end{align*}
which concludes the proof, as we can simply set $\eta(\{y\}):=1$ and $d:=\nu(\{y\})/2$.
\end{proof}
In particular, note that the second condition stated in 3. from Proposition~\ref{pr:assumption.relation} holds automatically if $E$ is denumerable, so that in this case \eqref{A.1'} and \eqref{A.2'} are effectively equivalent. That saying, note that condition \eqref{A.2} is stronger than \eqref{A.1}, see Example~\ref{ex:A2vsA1}.

\begin{example}[One-step Mixing does not imply one-step uniform ergodicity]\label{ex:A2vsA1}
Let $E=\{1,2,3\}$ and consider the transition matrix \[
P:=\begin{bmatrix}
0 & \tfrac 12 & \tfrac 12\\
\tfrac 12 & 0& \tfrac 12\\
\tfrac 12 & \tfrac 12 & 0
\end{bmatrix}.
\]
It can be easily checked that the transition kernel linked to matrix $P$ satisfies \eqref{A.1} with $\Lambda=\tfrac{1}{2}$ but does not satisfy \eqref{A.2} since zero-diagonal entries imply that the  minorization measure $\eta\in \cP(E)$ from $\eqref{A.2}$ must be non-positive for all states.\end{example}

For simplicity, in this paper we state most of the results only for \eqref{A.1} having it mind they are also true for \eqref{A.2} and could be easily generalised to multi-step framework by considering \eqref{A.1'} or \eqref{A.2'}.

\section{Local contraction property under ergodic assumptions}\label{S:local.contr}

In this section we state and prove an important result stating that for any $\bP\in \bC(E)$ satisfying assumption \eqref{A.1}, the operator $T$ defined in \eqref{lcont} is a local contraction in a suitable norm; this is an essential result for the span-contraction approach. For any $f\in C(E)$ we introduce the supremum norm $\|\cdot\|$ and the linked span semi-norm $\|\cdot\|_{\textrm{sp}}$ that are given by
\[
\|f\|:=\sup_{x\in E}|f(x)|\quad\textrm{and}\quad \|f\|_{\textrm{sp}}:=\frac{\sup_{x\in E} f(x)- \inf_{y\in E} f(y)}{2}.
\]
 Note that for any $f\in C(E)$ those two norms are linked by relation 
 \begin{equation}\label{eq:norm.rel}
 \inf_{c\in\bR}\|f+c\|=\|f\|_{\textrm{sp}},
 \end{equation}
 i.e. the span norm could be seen as the supremum norm for centered function $f$; see Proposition 2 in \cite{PitSte2016} for details. In particular, since the function $w$ in \eqref{eq:3:bellman.av} is defined up to an additive contant, it is often more convenient to use span semi-norm rather than the supremum norm. Moreover, for any signed measure $\mu:=\mu_1-\mu_2$, where $\mu_1,\mu_2\in\cP(E)$, we define the total variation norm of $\mu$ by
\[
\|\mu \|_{\textrm{var}}:=\frac{1}{2}\int_{E}|\mu|(dy)=\sup_{B\in\cE}|\mu_{1}(B)-\mu_{2}(B)|,
\]
where $|\mu|:=\1_{A}\mu - \1_{A^c}\mu$, and $A\in \cE$ is the positive set for $\mu$ that is obtained using Hahn-Jordan decomposition; see \cite{HerLas1996} for details. Also, we recall that for any bounded and measurable function $\varphi\colon E\to\bR$ and $\nu\in\cP(E)$, the corresponding entropic utility admits the dual (robust) representation, i.e. we get
\begin{equation}\label{eq:3:baseq}
\ln \left(\int_E e^{\varphi(x)}\nu(dx)\right)=\sup_{\mu \in {\cal{P}}(E)} \left[\int_E \varphi(x)\mu(dx) - \mathbb{H}[\mu\|\nu]\right],
\end{equation}
where, for any $\mu,\nu\in\cP(E)$,
\begin{equation}
\mathbb{H}[\nu\,\|\,\mu]=
\begin{cases}
\int_E \ln \left(\frac{d\nu}{d \mu}(x)\right)\nu(dx) & \textrm{if } \nu\ll \mu,\\
+\infty &\textrm{otherwise},
\end{cases}
\end{equation}
is the relative entropy of $\mu$ with respect to $\nu$. We are now ready to present a theorem which is a central tool of the risk sensitive span-contraction framework following \cite{DiMSte1999}.

\begin{theorem}\label{loccontr}
Assume $\bP\in \bC(E)$ satisfies \eqref{A.1} for $\Lambda\in (0,1)$ and let $g\in C(E)$. Then, operator $T$ defined in \eqref{lcont} is a local contraction in the span norm, i.e. there exists $\alpha\colon \bR_{+}\to (0,1)$ such that 
\[
\|Tf_1-Tf_2\|_{\textrm{sp}}\leq \alpha(M)\, \|f_1-f_2\|_{\textrm{sp}},
\]
for any $f_1,f_2\in \mathbb{C}(E)$ satisfying   $\|f_1\|_{\textrm{sp}}\leq M$  and $\|f_2\|_{\textrm{sp}}\leq M$. 
\end{theorem}
\begin{proof}
Let us fix $f_1,f_2\in C(E)$ and $x_1, x_2\in E$. For any $x\in E$ and $f\in C(E)$ we define the {\it Esscher transform} measure $\mu_{(x,f)}\in \cP(E)$ as
\begin{equation}\label{eq:3:mu.def}
\mu_{(x,f)}(B):=\frac{\int_B e^{f(z)}\mathbb{P}(x,dz)}{ \int_E
e^{f(z)}\mathbb{P}(x,dz)},\quad B\in \cE.
\end{equation}
Then, from the dual representation of entropic risk measure \eqref{eq:3:baseq} we get
\begin{align}
Tf_1(x_1) &\geq g(x_1) +\int_E f_1(y)\mu_{(x_1,f_2)}(dy)
-\mathbb{H}\left[ \mu_{(x_1,f_2)}\,\|\, \mathbb{P}(x_1,\cdot)\right],\nonumber\\
Tf_2(x_2) &\geq g(x_2) +\int_E f_2(y)\mu_{(x_2,f_1)}(dy)
-\mathbb{H}\left[ \mu_{(x_2,f_1)}\,\|\, \mathbb{P}(x_2,\cdot)\right].\label{inne}
\end{align}
Recalling \eqref{lcont} and taking into account that $\mu_{(x_1,f_2)}$ and $\mu_{(x_2,f_1)}$ are entropy maximising measures for parameter choices $(x_1,f_2)$ and $(x_2,f_1)$, respectively, we get 
\begin{align}
Tf_1(x_2) &\leq g(x_2) +\int_E f_1(y)\mu_{(x_2,f_1)}(dy)
-\mathbb{H}\left[ \mu_{(x_2,f_1)}\,\|\, \mathbb{P}(x_2,\cdot)\right],\nonumber\\
Tf_2(x_1) &\leq g(x_1) +\int_E f_2(y)\mu_{(x_1,f_2)}(dy)
-\mathbb{H}\left[ \mu_{(x_1,f_2)}\,\|\, \mathbb{P}(x_2,\cdot)\right].\label{inne2}
\end{align}
Thus, combining \eqref{inne} with \eqref{inne2}, and  setting $\mu :=\mu_{(x_2,f_1)}-\mu_{(x_1,f_2)}$, we get
\begin{equation}\label{eq:3:decomp1}
Tf_1(x_1)-Tf_2(x_1) -\bigl(Tf_1(x_2)-Tf_2(x_2)\bigr) \geq \int_E \bigl(f_1(y) -f_2(y)\bigr)\mu(dy).
\end{equation}
Let $\Gamma$ denote the Hahn-Jordan decomposition (positive) set for the signed measure $\mu$. Then,
\begin{align}
\int_E \bigl(f_1(y) -f_2(y)\bigr)\mu(dy) & \leq \sup_{y\in
E}\bigl(f_1(y)-f_2(y)\bigr)\mu(\Gamma) +\inf_{y\in E} \bigl(f_1(y) -f_2(y)\bigr)\mu(\Gamma^c)\nonumber\\
& \leq \sup_{y\in
E}\bigl(f_1(y)-f_2(y)\bigr)\mu(\Gamma) -\inf_{y\in E} \bigl(f_1(y) -f_2(y)\bigr)\mu(\Gamma)\nonumber\\
&=2\|f_1-f_2\|_{\textrm{sp}}\mu(\Gamma) =2\|f_1-f_2\|_{\textrm{sp}}\|\mu\|_{\textrm{var}}.\label{eq:3:decomp2}
\end{align}
Combining \eqref{eq:3:decomp1} with \eqref{eq:3:decomp2} and taking supremum over $x_1\in E$ and $x_2\in E$ we get
\begin{equation}
\|Tf_1-Tf_2\|_{\textrm{sp}}\leq \|f_1-f_2\|_{\textrm{sp}}\sup_{x,x'\in E}
\|\mu_{(x,f_1)}-\mu_{(x',f_2)}\|_{\textrm{var}},
\end{equation}
Now, we are going to show that for any $M>0$ we get
\begin{equation}\label{in1}
\alpha(M):=\sup_{\substack{f_1\in C(E):\\ \|f_1\|_{\textrm{sp}\leq M}}}\sup_{\substack{f_2\in C(E):\\ \|f_2\|_{\textrm{sp}\leq M}}} \left(\sup_{x,x'\in E}
\|\mu_{(x,f_1)}-\mu_{(x',f_2)}\|_{\textrm{var}}\right) <1.
\end{equation}
Assuming that \eqref{in1} is not true, there exists a sequence of objects $(f_{1n},f_{2n},B_n,x_n,x'_n)$, $n\in\bN$, where $f_{1n},f_{2n}\in C(E)$, $\|f_1\|_{\textrm{sp}}\leq M$, $\|f_2\|_{\textrm{sp}}\leq M$, $B_n\in \cE$, and $x_n,x'_n\in E$ are such that
\[
\left(\mu_{(x_n,f_{1n})}-\mu_{(x'_n, f_{2n)}}\right)(B_n)\to
1\quad\hbox{as}\quad n\to\infty.
\]
In particular, this implies $\mu_{(x_n, f_{1n})}(B^c_n)\to 0$
and $\mu_{(x'_n, f_{2n})}(B_n)\to 0$, as $n\to\infty$. Now, from definition \eqref{eq:3:mu.def} it follows that
\[
e^{-\|f\|_{\textrm{sp}}}\mathbb{P}(x,B)\leq\mu_{(x,f)}(B),
\]
which in turn implies $\mathbb{P}(x_n,B^c_n)\to 0$ and $\mathbb{P}(x'_n,B_n)\to 0$, as $n\to\infty$.
Consequently, we get
\[
\lim_{n\to\infty} \left[\mathbb{P}(x_n,B_n) -\mathbb{P}(x'_n,B_n)\right]=1,
\]
which directly contradicts assumption \eqref{A.1}.
\end{proof}
\begin{remark}[Local contraction for bounded measurable functions]
Note that in general we also have a local contraction property in the space $\mathbb{B}(E)$ of bounded measurable functions with supremum norm. Furthermore, the constant $\alpha(M)$ does not depend on $g\in \mathbb{B}(E)$.
\end{remark}
From Theorem~\ref{loccontr} one can conclude that if the sequence of iterated MPE operators $(T^n0)_{n\in\bN}$ is bounded in the span norm, then we can apply standard Banach's fixed point arguments to get solution to \eqref{eq:3:bellman.av}. Unfortunately, this is not always the case, i.e. we could get
\[
\sup_{n\in\bN}\|T_n(0)\|_{\textrm{sp}}=\infty
\]
and there might be no bounded solution to \eqref{eq:3:bellman.av} under \eqref{A.1}. For completeness, let us introduce a simplified discrete state space example.

\begin{example}\label{ex:3:ergodic.not} (No generic solution to MPE under ergodicity assumption)
Let $E:=\{x_1,x_2\}$, $\Lambda\in (0,1)$, and let the transition matrix by given by
\[
P=\begin{bmatrix}
1 & 0\\
1-\Lambda & \Lambda
\end{bmatrix}.
\]
Then, the solution to MPE for $g\in C(E)$ exists if and only if $g(x_1)>g(x_2)$ or 
$\|g\|_{\textrm{sp}}<-\ln \sqrt{\Lambda}$.
\end{example}
\begin{proof}
Note that both \eqref{A.1} and \eqref{A.2} are satisfied. Under the assumed dynamics, the MPE could be restated as
\begin{equation}\label{eq:ex1}
\begin{cases}
w(x_1)+\lambda=g(x_1)+w(x_1),\\
w(x_2)+\lambda=g(x_2)+\ln\left[ (1-\Lambda) e^{w(x_1)}+\Lambda e^{w(x_2)}\right].
\end{cases} 
\end{equation}
Easy algebraic check shows that \eqref{eq:ex1} is equivalent to 
\begin{equation}\label{eq:ex1.1}
\begin{cases}
g(x_1)=\lambda,\\
g(x_1)-g(x_2)=\ln\left[ (1-\Lambda) e^{w(x_1)-w(x_2)}+\Lambda\right].
\end{cases} 
\end{equation}
From \eqref{eq:ex1.1} we see that if $g\in C(E)$ is such that $g(x_1)-g(x_2)<\ln \Lambda$  then no solution to MPE exists due to monotonicity of the logarithm function. In other words, the solution to MPE exists if and only if $g(x_1)>g(x_2)$ or 
\begin{equation}\label{eq:d.pre}
\|g\|_{\textrm{sp}}<-\ln \sqrt{\Lambda}.
\end{equation}
To illustrate how this impacts the values of $\|T^n(0)\|_{\textrm{sp}}$ let us set $\Lambda=\tfrac{1}{2}$, $g(x_1)=0$ and $g(x_2)=\ln k$, for some $k>0$. Then, we get
\[
\|T^{n}(0)\|_{\textrm{sp}}  =\frac{|T^n(0)(x_1)-T^n(0)(x_2)|}{2}=\left|0+\frac{1}{2}\ln\left[\frac{1}{2}+\frac{1}{2}\sum_{i=1}^{n}\left(\frac{k}{2}\right)^i\right]\right|.
\]
Clearly, if $k\geq 2$ then $\|T^{n}(0)\|_{\textrm{sp}}\to\infty$, as $n\to\infty$. On the other hand, if $k<2$, then the we get $\sup_{n\in\bN}\|T^{n}(0)\|_{\textrm{sp}}<\infty$ and the solution exists. For completeness, let us also show how this interacts with the local contraction property. For any $f_1,f_2\in C(E)$, such that $\|f_1-f_2\|_{\textrm{sp}}\neq 0$, we get
\begin{align*}
\|Tf_1-Tf_2\|_{\textrm{sp}} &=
\frac{|(f_1(x_1)-f_2(x_1)) - \left(\ln\left[ \tfrac{1}{2}e^{f_1(x_1)}+\tfrac{1}{2}e^{f_1(x_2)}\right]-\ln\left[ \tfrac{1}{2}e^{f_2(x_1)}+\tfrac{1}{2}e^{f_2(x_2)}\right]\right)|}{2}\\
& =
\frac{\left|(f_1-f_2)(x_1)-(f_1-f_2)(x_2)+\ln\left[\frac{ 1+e^{f_2(x_1)-f_2(x_2)}}{1+e^{f_1(x_1)-f_1(x_2)}}\right]\right|}{2}\\
&=\left|1+\frac{\ln\left(\frac{1+e^{M_2}}{1+e^{M_1}}\right)}{M_1-M_2}\right| \cdot \|f_1-f_2\|_{\textrm{sp}},
\end{align*}
where $M_1:=f_1(x_1)-f_1(x_2)$, $M_2:=f_2(x_1)-f_2(x_2)$, and $\|f_1-f_2\|_{\textrm{sp}}=\tfrac{1}{2}|M_1-M_2|$. Consequently, it is easy to show that for $f_1,f_2\in C(E)$ satisfying $\|f_1\|_{\textrm{sp}}\leq M$ and $\|f_2\|_{\textrm{sp}}\leq M$, we get the local contraction property e.g. for (local) shrinkage constant
\[
\alpha(M)=1/(1+e^{-M}).
\]
Indeed, without loss of generality let us assume that $M_1-M_2>0$. Then, it is sufficient to show that
\begin{equation}\label{eq:ex1:add}
-(\alpha(M)+1)(M_1-M_2)<\ln\left(\tfrac{1+e^{M_2}}{1+e^{M_1}}\right)<(\alpha(M)-1)(M_1-M_2).
\end{equation}
First, let $h(x):=\ln(1+e^x)-(\alpha(M)+1)\cdot x$. Then, left inequality in \eqref{eq:ex1:add} could be rewritten as $h(M_1)<h(M_2)$. Now, since for $x\leq M$ we have
\[
h'(x)=\frac{e^x}{1+e^x} -(\alpha(M)+1) \leq \frac{e^{M}}{1+e^{M}} -(\alpha(M)+1)< -\alpha(M)<0,
\]
we get left inequality in \eqref{eq:ex1:add}. Second, set $k(x):=\ln(1+e^x)+(\alpha(M)-1)\cdot x$. Then, the right inequality in \eqref{eq:ex1:add} could be rewritten as $k(M_1)>k(M_2)$. Noting that for $x\geq -M$ we have
\[
k'(x)=\frac{e^x}{1+e^x} +\alpha(M)-1 \geq \frac{e^{-M}}{1+e^{-M}} +\alpha(M)-1=0,
\]
we conclude the proof of \eqref{eq:ex1:add}.
\end{proof}

Following the proof of Example~\ref{ex:3:ergodic.not} we see a fundamental difference between risk-sensitive and risk-neutral control, where global minorization is typically sufficient to guarantee the existence of the associated Additive Poisson Equation (APE), see Remark~\ref{rem:fund}.

\begin{remark}[Fundamental difference between risk neutral and risk sensitive framework]\label{rem:fund}
Under \eqref{A.1}, the additive Poisson operator $\tilde Tf(x):=g(x)+\bE_x(f)$, $f\in C(E)$, is contractive in the span norm so that there is a unique (up to an additive constant) solution to APE; see Section 5.2.2 in~\cite{HerLas1996} for details. This points out to a fundamental difference between risk neutral and risk sensitive frameworks. To illustrate this,  let us assume the dynamics introduced in Example~\ref{ex:3:ergodic.not} and apply it to risk-neutral framework. Similarly as in \eqref{eq:ex1}, the APE could be stated as
\[
\begin{cases}
w(x_1)+\lambda=g(x_1)+w(x_1),\\
w(x_2)+\lambda=g(x_2)+\left[ (1-d)w(x_1)+d w(x_2)\right].
\end{cases} 
\]
It is easy to check that this equation has a solution for any $g\in C(E)$ and no additional restriction is required here. Indeed, the exemplary solution is given by $\lambda:=g(x_1)$, $w(x_1):=(1-d)^{-1}[g(x_1)-g(x_2)]$, and $w(x_2):=0$.
\end{remark}

\begin{remark}[Necessary and sufficient conditions for MPE existence on finite state space]
Assuming that the state-space is finite, one could derive necessary and sufficient conditions for MPE equation solution existence for any $g\in C(E)$ in both risk-neutral and risk-sensitive setting. In the finite state space risk-sensitive case, MPE solution exists if and only if {\it Unichain condition} and {\it Strong Doeblin condition} is satisfied. Again, this is essentially different from risk-neutral case, where APE exists if and only if {\it Unichain condition} is satisfied. We refer \cite{CavHer2009} for more details.
\end{remark}
\begin{remark}[Split probability condition]
In \cite{DiMSte2007}, the split probability space technique is used to get explicit formula for the solution of MPE. One can show that sufficient conditions for this formula existence are also necessary in the case of Example \ref{ex:3:ergodic.not}. We refer to \cite{DiMSte2007} for more details.
\end{remark}

\section{Existence of solution to MPE for predetermined reward function}\label{S:predReward}

In this section we answer the question when \eqref{A.1} is sufficient to guarantee the existence of a solution to MPE for any $\bP\in\bC(E)$, where the function $g\in C(E)$ is predetermined. From Example~\ref{ex:3:ergodic.not}, we infer that the size of the span-norm of $g$ plays a key role. Next result shows that this is indeed the case and one can find a sharp bound imposed on $\|g\|_{\textrm{sp}}$. For brevity, we introduce function
\[
k(x):=-\tfrac{1}{2}\ln x,\quad x>0.
\]
and show the necessary and sufficient conditions for MPE existence under \eqref{A.1} using some arguments of \cite{DiMSte2000}.
\begin{theorem}\label{th:sharp1}
Let us fix $g\in C(E)$ and let $\Lambda\in (0,1)$. Then, the solution to MPE exists for any $\bP\in\bC(E)$ satisfying \eqref{A.1} with $\Lambda$ if and only if  $\|g\|_{\textrm{sp}} < k(\Lambda)$.
\end{theorem}
\begin{proof}
($\Leftarrow$) Let $g\in C(E)$ be such that  $\|g\|_{\textrm{sp}} < k(\Lambda)$ and let $\bP\in\bC(E)$ satisfy \eqref{A.1} for $\Lambda$. First, Let us show that the iterated sequence of MPE operators is bounded in the span norm, i.e. there exists $M\in\bR$ such that $\| T^n0\|_{\textrm{sp}}<M$ for $n\in\bN$. From \eqref{A.1},  for  any non-negative function $f\in C(E)$ and $x,x'\in E$ we  get
\begin{align}
&\mathbb{P}f(x)-\mathbb{P}f(x')\leq \int_E f(y)(\mathbb{P}(x,dy)-\mathbb{P}(x',dy))\leq \nonumber \\
&\int_H f(y)(\mathbb{P}(x,dy)-\mathbb{P}(x',dy))\leq \sup_{y\in E}f(y) \Lambda 
\end{align}
where $H$ comes for Hahn decomposition of $\mathbb{P}(x,dy)-\mathbb{P}(x',dy)$. Therefore, we have 
\begin{align}
Tf(x)-Tf(x') & \leq g(x)-g(x')+ \ln \frac{\int_E e^{f(y)}\mathbb{P}(x,dy)}{\int_E e^{f(y)}\mathbb{P}(x',dy)}\nonumber \\
&\leq  2\|g\|_{\textrm{sp}} + \ln \frac{\int_E e^{f(y)}\mathbb{P}(x',dy)+\Lambda e^{\sup_{y\in E} f(y)}}{\int_E e^{f(y)}\mathbb{P}(x',dy)}\nonumber \\
& \leq  2\|g\|_{\textrm{sp}} + \ln \left[1+\Lambda e^{2\|f\|_{\textrm{sp}}} \right]
\end{align}
which implies $\|Tf\|_{\textrm{sp}}\leq \|g\|_{\textrm{sp}} + \tfrac{1}{2}\ln\left[1+\Lambda e^{2\|f\|_{\textrm{sp}}}\right]$. Thus, by iteration, for $n\in\bN$, we get
%
%\begin{align*}
%\|T^nf\|_{\textrm{sp}}& \leq \|g\|_{\textrm{sp}} + \frac{1}{2}\ln\left[1+\Lambda e^{2\|T^{n-1}\|_{\textrm{sp}}}\right]\\
%& \leq \|g\|_{\textrm{sp}} + \frac{1}{2}\ln\left[1+\Lambda e^{2\|g\|_{\textrm{sp}}}\left(1+\Lambda e^{2\|T^{n-2}\|_{\textrm{sp}}}\right)\right]\\
%& \leq \|g\|_{\textrm{sp}} + \frac{1}{2}\ln\left[1+\Lambda e^{2\|g\|_{\textrm{sp}}}\left(1+\Lambda e^{2\|g\|_{\textrm{sp}}}\left(1+ \Lambda e^{2\|T^{n-3}\|_{\textrm{sp}}}\right)\right)\right]
%\end{align*}
%
\begin{equation}
\|T^n f\|_{\textrm{sp}}\leq \|g\|_{\textrm{sp}} + \frac{1}{2}\ln\left[1+ \sum^{n-1}_{i=1}\left(\Lambda e^{2\|g|\|_{\textrm{sp}}}\right)^i + \Lambda^n e^{(n-1)2\|g\|_{\textrm{sp}}+2\|f\|_{\textrm{sp}}}\right].
\end{equation}
Since $\|g\|_{\textrm{sp}}<k(\Lambda)$, we know that $\Lambda e^{2\|g\|_{\textrm{sp}}}<1$. Therefore, using geometric convergence arguments and setting $\epsilon_n:=\Lambda^n e^{(n-1)2\|g\|_{\textrm{sp}}+2\|f\|_{\textrm{sp}}}$, we get
\begin{equation}
\|T^nf\|_{\textrm{sp}}\leq \|g\|_{\textrm{sp}} + \frac{1}{2}\ln\left[1+ \frac{\Lambda e^{2\|g\|_{sp}} }{ 1-\Lambda e^{2\|g\|_{\textrm{sp}}}}+\epsilon_n\right]=\|g\|_{\textrm{sp}}+\frac{1}{2}\ln\left(\frac{1}{1-\Lambda e^{2\|g\|_{\textrm{sp}}}}+\epsilon_n\right).
\end{equation}
Now, since $\epsilon_n\searrow 0$, $n\to\infty$, and $1-\Lambda e^{2\|g\|_{\textrm{sp}}}>0$ we know that for sufficiently big $n\in\bN$ we get finite upper bound for $\|T^nf\|_{\textrm{sp}}$ that does not depend on $f$. In particular, this holds for $f\equiv 0$, which concludes the proof of the fact that the sequence $\|T^n 0\|_{\textrm{sp}}$, $n\in\bN$, is bounded. Now, combining this result with Theorem~\ref{loccontr} we can apply standard Banach's fixed point arguments to get solution to MPE equation \eqref{eq:3:bellman.av}; see e.g. \cite{PitSte2016}. This concludes this part of the proof.

\medskip

\noindent ($\Leftarrow$)
We want to show that for any $g\in C(E)$  such that $\|g\|_{\textrm{sp}} \geq k(\Lambda)$ one can find kernel $\bP\in\bC(E)$ that satisfies \eqref{A.1} with $\Lambda$ but for which a solution to MPE equation \eqref{eq:3:bellman.av} does not exists. 

First, let us assume that $g\in C(E)$ is such that $\|g\|_{\textrm{sp}} > k(\Lambda)$ or $\|g\|_{\textrm{sp}} = k(\Lambda)$ and the extremes are attained by $g$. For brevity, we only sketch the proof based on the analysis performed in Example~\ref{ex:3:ergodic.not}. In this case, one can find $x,x'\in E$, such that $g(x)-g(x')\geq k(\Lambda)$. Now, we can find a kernel $\bP\in \bC(E)$ with support on $\{x,x'\}$ such that it satisfies \eqref{A.1} for $\Lambda$, and for which
\[
\bP(x,x)=1,\quad \bP(x,x')=0,\quad \bP(x',x)=1-\Lambda,\quad \bP(x',x')=\Lambda.
\]
In consequence, after one iteration of operator $\bP$ we effectively end up with a situation similar to this introduced in Example~\ref{ex:3:ergodic.not} in which it has been shown that the solution to MPE exists if and only if $\|g\|_{\textrm{sp}} < k(\Lambda)$. This leads to contradiction with the fact that $g(x)-g(x')\geq k(\Lambda)$.

Second let us assume that $\|g\|_{\textrm{sp}} = k(\Lambda)$ and the upper extreme is not attained. Let $x_1\in E$ be such that $x_1=\inf_{x\in E}g(x)$ and let $(x_i)_{i=2}^{\infty}$ denote a sequence of points such that $(g(x_i))_{i=2}^{\infty}$ is increasing and $g(x_i)\to \sup_{x\in E}g(x)$ as $i\to\infty$.  Let $\bP\in \bC(E)$ denote the kernel with support on $(x_i)_{i=1}^{\infty}$ such that for $i=2,3,\ldots$, we have
\[
\bP(x_1,x_1)=1,\quad \bP(x_1,x_i)=0,\quad \bP(x_i,x_1)=1-\Lambda,\quad \bP(x_i,x_{i})=\Lambda.
\]
Assuming the existence of bounded MPE, and using similar logic as in Example~\ref{ex:3:ergodic.not}, we get
\[
\begin{cases}
g(x_1)=\lambda,\\
g(x_1)-g(x_i)=\ln\left[ (1-\Lambda) e^{w(x_1)-w(x_i)}+\Lambda\right].
\end{cases} 
\]
Now, noting that $g(x_1)-g(x_i)\to -2\|g\|_{\textrm{sp}}$, as $i\to\infty$, and $-2\|g\|_{\textrm{sp}}=-2k(\Lambda)=\ln\Lambda$, we get that $w(x_i)\to \infty$ as $i\to\infty$, which contradicts the fact that $w\in C(E)$.

The remaining cases when both extremes are not attained, or only the lower extreme is attained could be treated with similar logic; we omit the proof for brevity.
\end{proof}

For completeness, we state similar result
under \eqref{A.2} assumption and provide alternative proof of the underlying operator boundedness, based directly on span-norm size analysis.

\begin{theorem}\label{pr:sharp2}
Let us fix $g\in C(E)$ and let $d\in (0,1)$. Then, the solution to MPE exists for any $\bP\in\bC(E)$ satisfying \eqref{A.2} with $d$ if and only if $\|g\|_{\textrm{sp}} < k(1-d)$.
\end{theorem}

\begin{proof} 
The proof of this fact follows directly from Theorem \ref{th:sharp1} combined with Proposition \ref{pr:assumption.relation} since \eqref{A.2} implies \eqref{A.1}. Nevertheless, we decided to present an alternative proof of the fact that under \eqref{A.2} with $d$, the iterated operator $T$ is bounded for any $g\in C(E)$ such that $\|g\|_{\textrm{sp}}<k(1-d)$. The proof is based on a span norm centering and distance analysis that is aligned with the method introduced in the proof of Theorem~\ref{loccontr}. Let $g\in C(E)$ be such that  $\|g\|_{\textrm{sp}} < k(1-d)$ and let $\bP\in\bC(E)$ satisfy \eqref{A.2} for $d$. For $n\in\bN$, let $g_n:=T^n0$ and $c_n:=\inf_{c\in\bR}\|g_n+c\|$. Then, we get
\begin{align}
\| g_{n+1}\|_{\textrm{sp}} & = \sup_{x,y\in E}\frac{g(x)+\mu_x(g_n)-g(y)-\mu_y(g_n)}{2}\nonumber\\
&\leq \|g\|_{\textrm{sp}}+\sup_{x,y\in E}\frac{\mu_x(g_n+c_n)-\mu_y(g_n+c_n)}{2}.\label{eq:span1}
%& \leq \delta\|f\|+\tfrac{1}{2}\|g_n\|_{\textrm{sp}}-\tfrac{1}{2}\inf_{y\in E}\mu_y(g_n(X_\delta)+c_n)
\end{align}
From $\|g\|_{\textrm{sp}} < k(d)$ we know that there exists $\epsilon\in\bR$ such that $d>\frac{\epsilon}{ 2}>0$ and
\begin{equation}\label{eq:eps.fun}
\|g\|_{\textrm{sp}} <-\tfrac{1}{2}\ln(1-d+2\epsilon).
\end{equation}
Let $K:=\ln (d-\epsilon)-\ln \epsilon$ and $A_n:=\{x\in E\colon g_n(x)+c_n\leq 0\}$. For any fixed $n\in\bN$ we consider three disjoint cases: (1) $\|g_n\|_{\textrm{sp}}\leq K$; (2) $\|g_n\|_{\textrm{sp}}> K$ and $\inf_{x\in E}\bP(x,A^c_n)\geq \epsilon$; (3) $\|g_n\|_{\textrm{sp}}>K$ and $\inf_{x\in E}\bP(x,A^c_n)< \epsilon$. 

\noindent In the first case, noting that $\|g_n+c_n\|=\|g_n\|_{\textrm{sp}}$, and using \eqref{eq:span1}, we get 
\begin{equation}\label{eq:span.2a}
\|g_{n+1}\|_{\textrm{sp}}\leq K+\|g\|_{\textrm{sp}}.
\end{equation}
In the second case, using \eqref{eq:span1}
and the fact that $1_{{A_n}}(g_n(x)+c_n)\geq - 1_{{A_n}}\|g_n\|_{\textrm{sp}}$ and noting that for any $x,y\in E$ we have
\begin{align*}
\mu_y(g_n+c_n)& \geq \mu_y(-1_{A_n} \|g_n\|_{\textrm{sp}})\\
&= \ln\left[e^{-\|g_n\|_{\textrm{sp}}}\bP(y,A_n)+e^{0}\bP(y,A^c_n)\right]\\
&\geq \inf_{y\in E}\ln\bP(y,A^c_n)\\
&\geq \ln \epsilon
\end{align*}
and $\mu_x(g_n+c_n)\leq \|g_n\|_{\textrm{sp}}$, we get
\begin{equation}\label{eq:span.2b}
\| g_{n+1}\|_{\textrm{sp}}  \leq \tfrac{1}{2}\| g_{n}\|_{\textrm{sp}}+ \|g\|_{\textrm{sp}}-\tfrac{1}{2}\ln \epsilon.
\end{equation}
In the third case, we get $\sup_{x\in E}\bP(x,A_n)> 1-\epsilon$. Thus, recalling assumption \eqref{A.1}, we get $\epsilon\geq \inf_{x\in E}\bP(x,A_n^c)\geq d\nu(A_n^c)\geq d-\bP(x,A_n)$ so that 
$\inf_{x\in E}\bP(x, A_n)\geq d-\epsilon$. Consequently, for any $x,y\in E$ we have
\begin{align*}
\mu_x(g_n+c_n )&\leq \mu_x( 1_{A_n^c}\|g_n\|_{\textrm{sp}})\\
&\leq \|g_n\|_{\textrm{sp}}+\mu_x( -1_{A_n}K)\\
&= \|g_n\|_{\textrm{sp}}+\ln\left[e^{-K}\bP(x,A_n)+e^{0}\bP(x,A^c_n)\right]\\
&\leq \|g_n\|_{\textrm{sp}}+\ln\left[1-\left(1-\frac{\epsilon}{d-\epsilon}\right)\inf_{x\in E} \bP(x,A_n)\right]\\
&\leq \|g_n\|_{\textrm{sp}}+\ln\left[1-\left(\frac{d-2\epsilon}{d-\epsilon}\right)(d-\epsilon)\right]\\
&=\|g_n\|_{\textrm{sp}}+\ln\left[1-d+2\epsilon\right]
\end{align*}
and $\mu_y(g_n+c_n)\geq -\|g_n\|_{\textrm{sp}}$. Thus, recalling \eqref{eq:span1} and \eqref{eq:eps.fun} we get
\begin{equation}\label{eq:span.2c}
\| g_{n+1}\|_{\textrm{sp}} \leq \| g_{n}\|_{\textrm{sp}} +\|g\|_{\textrm{sp}} +\tfrac{1}{2}\ln\left[1-d+2\epsilon\right]\leq \| g_{n}\|_{\textrm{sp}}.
\end{equation}
Combining all three cases, i.e. \eqref{eq:span.2a}, \eqref{eq:span.2b}, and \eqref{eq:span.2c}, for any $n\in\bN$ we get
\[
\| g_{n+1}\|_{\textrm{sp}} \leq \max\left\{K+\|g\|_{\textrm{sp}},\,\, \tfrac{1}{2}\| g_{n}\|_{\textrm{sp}}+ \|g\|_{\textrm{sp}}-\tfrac{1}{2}\ln \epsilon,\,\,  \| g_{n}\|_{\textrm{sp}}\right\},
\]
from which the proof follows using standard geometric series arguments.
\end{proof}

Theorem~\ref{th:sharp1} sheds some light on the interaction between the risk-neutral and risk-sensitive framework that are represented by additive and multiplicative Poisson equations, respectively. The APE given by
\[
w_0(x)=g(x)-\lambda_0+\int_E w_0(y)\bP(x,dy),\quad x\in E,
\]
could be seen as a limit of risk-sensitive MPEs with risk-aversion $\gamma\neq 0$ given by
\[
w_{\gamma}(x)=\gamma g(x)-\lambda_{\gamma}+\ln\int_E e^{w_{\gamma}(y)}\bP(x,dy),\quad x\in E,
\]
for $\gamma\to 0$.  Note that by introducing risk averse parameter $\gamma$ we are effectively replacing $g$ by $\gamma g$ in the standardised MPE problem. Consequently, whatever the initial choice of $g\in C(E)$, we can find $\gamma$ small enough, i.e. such that the span-norm induced bound presented in 
Theorem~\ref{th:sharp1} (or Theorem~\ref{pr:sharp2}) is satisfied by $\gamma g$. This indicates that solution to APE should always exists under \eqref{A.1}, which is indeed the case, see e.g. \cite{KonMey2003}. Also, note that some results presented in the literature state existence of MPE (or its controlled version) for sufficiently small $\gamma$ which is consistent with the results presented herein, see e.g. Theorem 1 in \cite{Ste1999} or Theorem 5.4 in \cite{BalMey2000}.

%\[
%0=\mu_x\left(\sum_{i=1}^{T}(g(x_i)-\lambda)+w(x_T)\right)
%\]

\section{Existence of bounded MPE solution for generic reward function}\label{MPE.always}
In this section we investigate what additional assumptions, apart from \eqref{A.1}, could be imposed on a transition kernel $\bP\in \bC(E)$, so that the bounded solution to MPE exists for arbitrary $g\in C(E)$. It should be noted that while assumption \eqref{A.1} is natural for compact spaces, it might be not satisfied for general locally compact spaces. In fact, for non-compact state space, one can show that MPE solution is typically unbounded under very generic conditions linked to so called $C_0$-Feller property.

The typical condition imposed on MPE in the mixing framework relates to so called {\it strong mixing}. We say that the transition kernel satisfies {\it multi-step strong mixing condition} for $n\in\bN$ and $L>0$ if
\begin{equation}\tag{A.3}\label{A.3}
 \sup_{x,x'\in E}  \frac{\bP_n(x,B)}{\bP_n(x',B)}\leq L,\quad B\in\cE,
 \end{equation}
where the convention $\frac{0}{0}=0$ and $\frac{1}{0}=\infty$ is used. Note that under \eqref{A.3} we have that iterated measures $\bP_n(x,\cdot)$, for $x\in E$, are equivalent. Furthermore, \eqref{A.3} implies \eqref{A.1'} for the same $n\in\bN$,  which effectively leads to local contraction property of operator $T^n$; see Theorem~\ref{loccontr}. This condition is satisfied e.g. by regular reflected diffusions in bounded domains, see~\cite[Remark 2.1]{MenRob2018} and references therein. Let us now show that under \eqref{A.3} the iterated operator $T$ is bounded in the span norm, so that we get the existence of solution to MPE.

\begin{proposition}
Let $\bP\in \bC(E)$ satisfy \eqref{A.3}. Then, the solution to MPE equation \eqref{eq:3:bellman.av} exists for any $g\in C(E)$.
\end{proposition}
\begin{proof}
Due to Theorem~\ref{loccontr} it is sufficient to show that the iterated sequence of MPE operators is bounded, i.e. there exists $M\in\bR$ such that $\| T^n0\|_{\textrm{sp}}<M$ for $n\in\bN$. For any $f\in C(E)$ and $x,x'\in E$ we get 
\begin{align}
T_nf(x)-T_nf(x') & \leq 2n\|g\|_{\textrm{sp}}+ \ln \frac{\int_E e^{f(y)}\mathbb{P}_n(x,dy)}{\int_E e^{f(y)}\mathbb{P}_n(x',dy)}\nonumber  \\
 & \leq 2n\|g\|_{\textrm{sp}}+ \ln \left[L+\frac{\int_E e^{f(y)}\left[\mathbb{P}_n(x,dy)-L\mathbb{P}_n(x',dy)\right]}{\int_E e^{f(y)}\mathbb{P}_n(x',dy)}\right]\nonumber  \\
& \leq 2n \|g\|_{\textrm{sp}} + \ln L.
\end{align}

Thus, for any $f\in C(E)$ we get $\|T_nf\|_{\textrm{sp}}\leq   n\|g\|_{\textrm{sp}} + \tfrac{1}{2}\ln L$. In particular, this implies that the iterated sequence $(T^{kn}0)_{k\in\bN}$, is jointly bounded in the span norm, and so is $(T^{n}0)_{n\in\bN}$. This concludes the proof.
\end{proof}

While \eqref{A.3} guarantees the existence of MPE solution, it could be seen as restrictive and designed for compact state spaces. In Example~\ref{ex:strong.mixing} we show that in general case \eqref{A.3} is not a necessary condition; note that similar example could be constructed for a dynamics with full invariant measure support.

\begin{example}[Existence of MPE without multi-step strong mixing]\label{ex:strong.mixing}
Let $E=\bN$ be a denumerable state space and let us consider a Markov chain with transition matrix
\[
P:=\begin{bmatrix}
1  & 0 & 0 &0 &0 & \ldots\\
\tfrac{1}{2}& 0  &\tfrac{1}{2} &0 &0 & \ldots\\
\tfrac{3}{4}& 0  &0 &\tfrac{1}{4} &0& \ldots\\
\tfrac{7}{8}& 0  &0 &0 &\tfrac{1}{8}& \ldots\\
\ldots & \ldots&\ldots &\ldots &\ldots & \ldots
\end{bmatrix}.
\]
Then, assumption \eqref{A.3} is not satisfied but MPE solution exists for any $g\in C(E)$.
\end{example}

\begin{proof}
Note that $\bP\in\cP(E)$ corresponding to transition matrix $P$ satisfies \eqref{A.1} with $\Lambda=\tfrac{1}{2}$ but \eqref{A.3} is not satisfied since for any $n\in\bN$ we get $\bP_n(1,\{n+1\})=0$ and $\bP_n(2,\{n+1\})>0$. Let us fix $g\in C(E)$. For simplicity, assume that $g(1)=0$. It is sufficient to show that $\sup_{n\in\bN}\|T^n0\|_{\textrm{sp}}<\infty$. Since for any $n\in\bN$ and $i\in\bN$ we have $T^n0(i) \geq -\|g\| +\ln\tfrac{1}{2}$, it is sufficient to show that $(T^n0)$ is uniformly bounded from above. For $n\in\bN$, we have $T^n0(1)=0$ and consequently, for any $i=2,3,\ldots$, we get
\begin{align*}
T^n0(i) &= g(i) +\ln\left[ (1-2^{-i})+2^{-i}e^{T^{n-1}0(i+1)}\right]\\
&\leq \|g\| +\ln\left[1+2^{-1}e^{T^{n-1}0(i+1)}\right]\\
&\leq \|g\| +\ln\left[1+2^{-1}e^{\|g\|+\ln [1+2^{-2}e^{T^{n-2}0(i+2)}]}\right]\\
& \leq \|g\| +\ln\left[1+2^{-1}e^{\|g\|}+2^{-3}e^{\|g\|} e^{T^{n-2}0(i+2)}\right]\\
& \leq \|g\| +\ln\left[1+2^{-1}e^{\|g\|}+2^{-3}e^{2\|g\|}+2^{-6}e^{2\|g\|} e^{T^{n-3}0(i+3)}\right]\\
& \leq \|g\|+\ln \left[\sum_{k=0}^{\infty}2^{-k(k+1)/2}e^{k\|g\|} \right]\\
& \leq \|g\|+\ln \left[\sum_{k=0}^{\infty}e^{k[\|g\|-(k+1)/2 \cdot \ln 2]} \right].
\end{align*}
Noting that the sequence $(\|g\|-(k+1)/2\cdot\ln 2)$ does not depend on $n\in\bN$, is decreasing, and for sufficiently big $k\in\bN$ we have $(\|g\|-(k+1)/2\cdot\ln 2)<0$, we conclude that $\sup_{n\in\bN} \|T^n\|<\infty$ which implies $\sup_{n\in\bN} \|T^n\|_{\textrm{sp}}<\infty$ due to \eqref{eq:norm.rel}. This implies existence of MPE for $g$ due to Theorem~\ref{loccontr} and Banach's fixed point theorem.
\end{proof}

Next, we show necessary and sufficient conditions for MPE existence that need to hold on the complement of the invariant measure support. 
Please recall that for $\bP\in \bC(E)$ satisfying \eqref{A.1} we know that there exists a unique invariant measure; see \cite{Doob}; we use $\nu\in \cP(E)$ to denote this measure. Also, to ease the notation, we define the first set hitting time as 
\begin{equation}\label{eq:hitting.time}
\tau_B:=\inf\{n\in\bN: x_n\in B\}, \quad B\in\cE.
\end{equation}
In a nutshell, the process must escape rather quickly from any compact set that is outside of the support of the invariant measure; the escape time must be faster than geometric. This intuition is formalised in Proposition~\ref{pr:escape.geometric}.

\begin{proposition}\label{pr:escape.geometric}
Let $\bP\in\bC(E)$ satisfy \eqref{A.1} and let $\nu$ correspond to its invariant measure. Let $E_{\nu}$ be an invariant set and let us assume that a bounded MPE solution on $E_{\nu}$ exists for any $g\in C(E_{\nu})$. Then, the bounded MPE solution on $E$ exists for any $g\in C(E)$ if and only if
\begin{equation}\label{eq:tau.1}
\forall \alpha\in (0,1)\,\exists\, n\in\bN\colon \sup_{x\in E^c_{\nu}}\bP_x[\tau_{E_{\nu}}>n]\leq \alpha^n.
\end{equation}
\end{proposition}

\begin{proof}
The fact that the existence of bounded MPE solution on $E$ implies  condition~\eqref{eq:tau.1} will follow directly from Proposition~\ref{pr:no-solution}; for brevity, we skip the direct proof and focus on the opposite implication. Let $g\in C(E)$. Without loss of generality, we can assume that $g\not\equiv 0$ and $\inf_{x\in E}g(x)\geq 0$. Let $\alpha:=\exp(-2\|g-\lambda\|-1)$. From \eqref{eq:tau.1} we know that there exists $n_0\in\bN$ such that
\[
\sup_{x\in E^c_{\nu}}\bP_x[\tau_{E_{\nu}}>n_0]\leq \alpha^{n_0}.
\]
 Let $(w,\lambda)\in C(E_{\nu})\times \bR$ denote the MPE solution for $g$ on the invariant set $E_{\nu}$ such that $\inf_{x\in E_{\nu}}w(x)=0$. Note that $\lambda\geq 0$ since $g$ is non-negative. For $x\in E$, let
\[
w_0(x):=
\begin{cases}
w(x) & \textrm{ if } x\in E_{\nu}\\
0 & \textrm { if } x\not\in E_{\nu}.
\end{cases}
\]
and let $T_{\lambda}(\cdot)=T(\cdot)-\lambda$. First, recalling that $E_{\nu}$ is an invariant set, for any $x\in E_{\nu}$ we get $T_{\lambda}w_0(x)=w(x)$ and consequently $T^nw_0(x)=w(x)$, for $n\in\bN$. Thus, for any $k\in\bN$ and $x\in E$,  recalling that $\inf_{x\in E_{\nu}}w(x)=0$, we get
\begin{align}
T_{\lambda}^{(k+1)n_0}w_0(x) & \geq \inf_{x\in E}\mu_x\left(\sum_{i=0}^{n_0-1}(g(x_i)-\lambda)+T_{\lambda}^{kn_0}w_0(x_n)\right)\nonumber\\
& \geq -n_0\|g-\lambda\|+\inf_{x\in E}\left(\mu_x\left(\1_{\{\tau_{E_\nu}\leq n_0\}}w(x_n)+\1_{\{\tau_{E_\nu}> n_0\}}T_{\lambda}^{kn_0}w_0(x_n)\right)\right)\nonumber\\
& \geq -n_0\|g-\lambda\| +\inf_{x\in E}\left(\mu_x\left(\1_{\{\tau_{E_\nu}> n_0\}}T_{\lambda}^{kn_0}w_0(x_n)\right)\right)\nonumber\\
& \geq -n_0\|g-\lambda\| +\inf_{x\in E}\ln \left(\bP_x[\tau_{E_\nu}\leq n_0]e^0+\bP_x[\tau_{E_\nu}> n_0]e^{-\|T_{\lambda}^{kn_0}w_0 \|}\right)\nonumber\\
& \geq -n_0\|g-\lambda\| +\ln (1-\alpha^{n_0}).\label{eq:T.lambda.1}
\end{align}
On the other hand, for any $k\in\bN$ and $x\in E$, we get
\begin{align}
T_{\lambda}^{(k+1)n_0}w_0(x) & \leq n_0\|g-\lambda\|+\sup_{x\in E}\left(\mu_x\left(\1_{\{\tau_{E_\nu}\leq n_0\}}w(x_n)+\1_{\{\tau_{E_\nu}> n_0\}}\|T_{\lambda}^{kn_0}w_0\|\right)\right)\nonumber\\
&\leq n_0\|g-\lambda\| +\ln (e^{\|w\|}+\alpha^{n_0}e^{\|T_{\lambda}^{kn_0}w_0\|}).\label{eq:T.lambda.2}
\end{align}
Now, let $a:=e^{-n_0(\|g-\lambda\|+1)}$, and $C:=e^{n_0\|g-\lambda\|}(e^{\|w\|}+(1-\alpha^{n_0})^{-1})$. Then, combining \eqref{eq:T.lambda.1} with \eqref{eq:T.lambda.2}, noting that $a<1$ as well as $C<\infty$, and taking the exponent, for any $k\in\bN$, we get
\begin{align}
e^{\|T_{\lambda}^{(k+1)n_0}w_0\|} &\leq e^{n_0\|g-\lambda\|}\left(e^{\|w\|}+ \alpha^{n_0}e^{\|T_{\lambda}^{kn_0}w_0\|}\right)+e^{n_0\|g-\lambda\|-\ln(1-\alpha^{n_0})}\nonumber\\
&=e^{n_0\|g-\lambda\| +n_0(-2\|g-\lambda\|-1)}e^{\|T_{\lambda}^{kn_0}w_0\|}+C\nonumber\\
& \leq ae^{\|T_{\lambda}^{kn_0}w_0\|}+C\nonumber\\
& \leq \sum_{i=0}^{k}a^iC +a^{k}e^{\|T_{\lambda}^{n_0}w_0\|}\nonumber\\
& \leq \frac{C}{1-a}+e^{\|T_{\lambda}^{n_0}w_0\|}\label{eq:T.lambda.3}
\end{align}
Noting that $\|T^{n_0}_{\lambda}w_0\|$ is bounded (as $T_{\lambda}$ is a $C$-Feller operator), we conclude that
\[
\sup_{k\in\bN}\|T_{\lambda}^{kn_0}w_0\|<\infty
\]
which implies $\sup_{n\in\bN}\|T_{\lambda}^{n}w_0\|<\infty$ and consequently $\sup_{n\in\bN}\|T^{n}w_0\|_{\textrm{sp}}<\infty$. Using Theorem~\ref{loccontr} and Banach's fixed point theorem, we get there exists $\tilde w\in C(E)$ and $\tilde \lambda\in\bR$ that satisfies MPE for $g$ on $E$, which concludes the proof.
\end{proof}
Proposition~\ref{pr:escape.geometric} allows us to immediately recover necessary and sufficient conditions for (bounded) MPE solution existence for a finite state space. Namely, the probability kernel must be such that it has a unique invariant measure and all states which lie outside of its support must be transient with return probability equal to zero as otherwise condition \eqref{eq:tau.1} would not hold. Assuming that the number of states is equal to $N\in\bN$, this immediately implies that after at most $N$-steps the process will be in the invariant measure support with probability 1. We refer to \cite{CavHer2009} where this characterisation is provided and discussed.

\begin{remark}[Difference between finite and denumerable case]
It is worth to note, that Example~\ref{ex:strong.mixing} illustrates that finite state space necessary and sufficient conditions for MPE existance stated in \cite{CavHer2009} do not transfer directly to denumerable spaces since we do not get into the support of the invariant measure in finite number of steps (with probability one) while the MPE solution always exists.
\end{remark}

\section{Geometric propagation and negative existence results}\label{S:negative}
In this section we investigate how the geometric process propagation interacts with the existence of MPE solution for any $g\in C(E)$. While most results presented in this section do not require assumption~\eqref{A.1}, we will typically assume that $\bP\in\bC(E)$ is such that it has a unique invariant measure $\nu\in \cP(E)$; recall that invariant measure existence is effectively implied by \eqref{A.1}.

We start with a simple lemma which states that geometric stay in a subset of states for which the reward value is high set a lower bound on the optimal value. 

\begin{lemma}\label{lm:geom}
Let $g\in C(E)$ be non-negative and such that it has MPE solution $(w,\lambda)$. Moreover, assume there is $k>0$, $K\subset E$ and $\alpha\in (0,1)$ such that $g(x)\geq k$, $x\in K$, and $\sup_{x\in E}\bP[\tau_{K^c}>n]>\alpha^n$, $n\in\bN$. Then, we have $\lambda\geq k+\ln\alpha$.
\end{lemma}

\begin{proof}
Using MPE iteration, for any $n\in\bN$, we get
\[
n\lambda \geq \sup_{x\in E}\mu_x\left(\sum_{i=0}^{n-1}g(x_i)\right)+\|w\|\geq \ln \left[\alpha^{n}e^{nk}\right]+\|w\|=n(k+\ln \alpha)+\|w\|.
\]
Dividing both sides by $n$ and taking the limit as $n\to\infty$, we get $\lambda\geq k+\ln\alpha$.
\end{proof}

In Proposition~\ref{pr:escape.geometric} we stated that the process must escape rather quickly from any compact set that is outside of the support of the invariant measure. Here, we formulate a negative results linked to this fact.

\begin{proposition}\label{pr:no-solution}
Let $\bP\in\bC(E)$ have a unique invariant measure $\nu\in \cP(E)$. Assume there exists a compact set $K\in\cE$ for which
\begin{enumerate}
\item There is $\epsilon>0$, such that $\nu(K_{\epsilon})=0$, where $K_{\epsilon}:=\{x\in E\colon \rho(x,K)\leq \epsilon\}$; 
\item There is $h\in (0,1)$, such that $\sup_{x\in K}\bP_x\left[\tau_{K^c}> n \right]>h^n$, for $n\in\bN$.
\end{enumerate}
Then, there exists $g\in C(E)$ for which bounded solution to MPE does not exist.
\end{proposition}

\begin{proof}
Assume there is such $K\in\cE$ and let $\epsilon,h\in (0,1)$ denote the corresponding constants. Let $g\in C(E)$ be given by
\begin{equation}
g(x)=
\begin{cases}
\ln (2/h) & \textrm{if } \rho(x,K)=0,\\
\ln (2/h)\cdot(1-\rho(x,K)/\epsilon)& \textrm{if } \rho(x,K)\in (0,\epsilon)\\
0 & \textrm {if } \rho(x,K)\geq \epsilon.
\end{cases}
\end{equation}
Assume that a bounded MPE solution for $g$ exists; let us denote the solution by $w\in C(E)$ and $\lambda\in\bR$. Using iteration we get
\begin{equation}\label{eq:mp.at3.2}
n\lambda=\mu_x\left(\sum_{i=0}^{n-1}g(x_i)+w(x_n)-w(x)\right),\quad x\in E, n\in\bN.
\end{equation}
Using \eqref{eq:mp.at3.2}, recalling that $\nu(K_{\epsilon})=0$, and taking any $x_0\in E$ that is in the support of the invariant measure $\nu$, we get
\begin{align*}
n\lambda &\leq \inf_{x\in E}\mu_x\left(\sum_{i=0}^{n-1}g(x_i)\right)+2\|w\|_{\textrm{sp}}\\
&\leq \mu_{x_0}\left(\sum_{i=0}^{n-1}\1_{K_{\epsilon}^{c}}(x_i)g(x_i)\right)+2\|w\|_{\textrm{sp}}\\
& \leq 2\|w\|_{\textrm{sp}}.
\end{align*}
Thus, letting $n\to\infty$, we conclude that $\lambda\leq 0$. On the other hand, using \eqref{eq:mp.at3.2} and property $\sup_{x\in K}\bP_x\left[\tau_{K^c}> n \right]>h^n$, for $n\in\bN$, we get
\begin{align*}
n\lambda &\geq \sup_{x\in E} \mu_x\left(\sum_{i=0}^{n-1}g(x_i)\right) -2\|w\|_{\textrm{sp}}\\
& \geq \sup_{x\in K}\mu_x\left(\ln \tfrac{2}{h}\cdot\sum_{i=0}^{n-1}\1_{K}(x_i)\right)-2\|w\|_{\textrm{sp}}\\
&\geq \ln\left[e^{\ln 2/h \cdot (n-1)}\cdot \sup_{x\in K}\bP_x[\tau_{K^c}> n-1]\right]-2 \|w\|_{\textrm{sp}}\\
&\geq \ln \frac{2^{n-1}h^{n-1}}{h^{n-1}}-2\|w\|_{\textrm{sp}}\\ &=(n-1) \ln 2-2\|w\|_{\textrm{sp}}.
\end{align*}
Thus, letting $n\to\infty$ we get $\lambda\geq\ln 2$, which leads to contradiction.
\end{proof}

%Let us now present a series of results which show how the geometric propagation interacts with MPE solution existence. It should be noted that while we pre-assumed condition \eqref{A.1} for $\bP$, the results are in fact true under slightly weaker conditions linked to the existence of the unique invariance measure $\mu \in \cP(E)$ and weak convergence of the Cesaro averages of the iterated transition probabilities to $\mu$.

In the remaining part of this section let us show a series of more technical results which illustrate why the MPE solution is typically unbounded under very generic conditions linked to so called $C_0$-Feller property and how it interacts with mixing, i.e. assumption \eqref{A.1'}. To do this, we need some additional notation that will allow us to study dynamic propagation on an arbitrary compact ball. For any $x\in E$ and $\epsilon\in\bR_+$ let
\[
B(x,\epsilon):=\left\{y\in E: \rho(x,y)\leq \epsilon\right\}.
\]
With slight abuse of notation, given a fixed $\bar x\in E$ and $\eta>1$ till the end of this section we will use notation
\begin{align*}
B & :=B(\bar x, \eta),\\
g_m(x) &:=[m\cdot \rho\left(x,B(\bar{x},\eta-\tfrac{1} {m})\right)]\wedge 1, \quad m\in\bN.
\end{align*}
and use $\lambda_m \in\bR$ and $w_m\colon E\to \bR$ to denote a generic (i.e. not necessarily bounded) solution to MPE equation  \eqref{eq:3:bellman.av} for $g_m$ assuming it exists. Note that $g_m\in C(E)$ and for any $x\in E$ and $m\in\bN$ we know that $1\geq g_m(x)\geq \1_{\bar{B}(\bar{x},\eta)}(x)$, $g_m(\cdot)$ is decreasing, and $g_m(x) \searrow \1_{\bar{B}^c(\bar{x},\eta)}(x)$ as $m\to\infty$. Also, recall that $\tau_B$ denotes the first hitting time to $B$ for the underlying process, see \eqref{eq:hitting.time}.

\begin{proposition}\label{pr:st1} Fix $\bar x\in E$, $\eta>1$, and $m\in\bN$. Assume $g_m\in C(E)$ admits bounded MPE solution $(w_m,\lambda_m)$. Then, we have
\begin{align}
\sup_{x\in E}\mu_x(\tau_{B}(1-\lambda_m)) &\leq 2\|w_m\|+1,\label{eq:ne2}\\
\sup_{x\in E} \bP_x\left[\tau_{B}>n\right] & \leq e^{2\|w_m\|+1}e^{-(1-\lambda_m)n}.\label{eq:ne2'}
\end{align} 
\end{proposition}

\begin{proof}
First, we show \eqref{eq:ne2}. Iterating MPE equation \eqref{eq:3:bellman.av}, for any $n\in\bN$ and $x\in E$, we get
\begin{equation}\label{eq:ne3}
 w_m(x)=\mu_x\left(\sum_{i=0}^{\tau_B\wedge n-1} (g_m(x_i)-\lambda_m)+w_{m}(x_{\tau_{B}\wedge n})\right).
 \end{equation}
Consequently, letting $n\to \infty$, using Fatou lemma, and monotonicity of entropic utility, for any $x\in E$ we have
\begin{equation}
  w_m(x)+\1_{B}(x)\geq \mu_x\left(\tau_{B}(1-\lambda_m)+w_m(x_{\tau_{B}})\right).
\end{equation}
which implies~\eqref{eq:ne2}. Equality \eqref{eq:ne2'} follows directly from~\eqref{eq:ne2}, which concludes the proof. 
\end{proof}
In the next corollary we show, that Proposition~\ref{pr:st1} implies that the MPE solutions are typically non-bounded under some additional assumptions.
\begin{corollary}\label{cor:st1}
Fix $\bar x\in E$, $\eta>1$, and $m\in\bN$. Assume $g_m\in C(E)$ admits generic MPE solution $(w_m,\lambda_m)$, and we have
\begin{align}
\lambda_m &<1,\label{eq:ne4b}\\
\sup_{n\in\bN}\inf_{x\in E} \bP_{x}\left[\exists_{s\leq  n}\ \  x_s\in B\right] & =0.\label{eq:ne4}
\end{align}
Then, the function $w_m$ must be unbounded.
\end{corollary}
\begin{proof}
Assume that $w_m$ is bounded and \eqref{eq:ne4b} holds. Let (large) $n\in\bN$ be such that
\[
e^{2\|w_m\|+1}e^{-(1-\lambda_m)n}<1.
\]
Then, directly from \eqref{eq:ne2'}, we get $\sup_{x\in E} \bP_{x}\left[\tau_B>n\right]<1$, and consequently
\[
    \inf_{x\in E} \bP_{x}\left[
    \exists_{s\leq  n}\ \  x_s\in B\right] =1-\sup_{x\in E} \bP_{x}\left[\tau_B>n\right]>0,
\]
which contradicts \eqref{eq:ne4}.
\end{proof}

Property~\eqref{eq:ne4} from Corollary~\ref{cor:st1} could be seen as a weaker version of the well known assumption in locally compact spaces saying the transition operator $\bP$ transforms the space $C_0(E)$ of continuous bounded functions vanishing at infinity into itself.
Also, Property~\eqref{eq:ne4b} does not seem to be strong as it holds under mild assumptions related to the invariant measure convergence; see Proposition~\ref{pr:st12}.

\begin{proposition}\label{pr:st12}
Fix $\bar x\in E$, $\eta>1$, and $m\in\bN$. Assume $g_m\in C(E)$ admits generic MPE solution $(w_m,\lambda_m)$ and $\bP$ has unique invariant measure $\nu$. Moreover, assume that $\nu(g_m):=\int_E g_m d \nu<1$, and there exists $x\in E$ such that for any $\epsilon>0$ there is $p>0$ such that for a sufficiently large $k\in\bN$ we have
\begin{equation}\label{eq:ne6}
\textstyle\bP_x\left[\left|\frac{1}{k} \sum_{i=0}^{k-1}g_m(x_i)-\nu(g_m)\right|\geq \epsilon\right]\leq e^{-kp}. 
\end{equation}
Then, we have $\lambda_m<1$.  
\end{proposition}
\begin{proof}
Let $x\in E$ be such that \eqref{eq:ne6} is satisfied and let $\epsilon>0$ be such that $\epsilon< 1-\nu(g_m)$. Then, there exists $p>0$ such that for sufficiently big $k\in\bN$, we have
\begin{align}
\bE_x\left[e^{\sum_{i=0}^{k-1}g_m(x_i)}\right] & \leq  e^k \bP_x\left[\left|\frac{1}{k} \sum_{i=0}^{k-1}g_m(x_i)-\nu(g_m)\right|\geq \epsilon\right] +e^{k(\nu(g_m)+\epsilon)} \nonumber \\
&\leq e^{k(1-p)}+e^{k(\nu(g_m)+\epsilon)}\label{eq:ne61}.
\end{align}
Without loss of generality, we can assume that $p<1-\nu(g_m)-\epsilon$; note that if \eqref{eq:ne61} holds for some $p>0$, than it also holds for any smaller positive $p$. Consequently, we get
\begin{equation}
\mu_x\left(\sum_{i=0}^{k-1}g_m(x_i)\right)=\ln \bE_x\left[e^{\sum_{i=0}^{k-1}g_m(x_i)}\right] \leq \ln (2e^{k(1-p)})=\ln 2 + k(1-p)   
\end{equation}
which implies $\lambda_m\leq 1-p$. 
\end{proof}
\begin{remark}[Estimates of the empirical measures for Feller Markov process]
Assumption \eqref{eq:ne6} follows from the estimates for empirical measures of Feller Markov processes. In the form we use here sufficient conditions can be found in Theorem 3 and its proof of  \cite{DuncanSt2000} and the theory comes back to famous papers \cite{DonskerVI} and \cite{DonskerVIII}.    
\end{remark}

In the next result we show that the mixing condition combined with \eqref{eq:ne4} leads to a compact-type dynamics for which we do not get condition~\eqref{eq:ne4b}.

\begin{proposition}
Fix $\bar x\in E$, $\eta>1$, and $m\in\bN$. Assume $g_m\in C(E)$ admits generic MPE solution $(w_m,\lambda_m)$ and $\bP$ satisfies \eqref{A.1'} and \eqref{eq:ne4}. Then, $\lambda_m=1$.
\end{proposition}
\begin{proof}
Fix $\bar x\in E$, $\eta>1$, and $m\in\bN$ such that $g_m$ admits MPE solution $(w_m,\lambda_m)$.
From  \eqref{A.1'}, it follows that the value $\nu(B)$, where $\nu$ is the unique invariant measure, is approximated from below by an increasing sequence $m_n(B):=\inf_{y\in E} \bP_n(y,B)$, for details see the proof of the case (b) in Section 5, Chapter 5 of~\cite{Doob}.

Now, since for any $x\in E$ we have $m_n(B)\leq \bP_n(x,B)$ and $\bP_n(x,B) \to 0$ as $n \to \infty$ due to \eqref{eq:ne4}, we get $\nu(B)=0$. Consequently, using Jensen inequality, we get
\begin{align}
\lambda_m &\textstyle =\lim_{n\to \infty} \tfrac{1}{n}\ln \bE_x\left[e^{\sum_{i=0}^{n-1} g_m(x_i)}\right] \nonumber \\
&\textstyle \geq \liminf_{n\to \infty}  \tfrac{1}{n}\ln \bE_x\left[e^{\sum_{i=0}^{n-1} \1_{B^c}(x_i)}\right] \nonumber \\
&\textstyle \geq \liminf_{n\to \infty} \tfrac{1}{n} \bE_x\left[\sum_{i=0}^{n-1} \1_{B^c}(x_i)\right]\nonumber\\
& =\nu(B^c)=1\nonumber,
\end{align}
which concludes the proof as $1\geq \lambda_m$ due to the fact that $1\geq g_m(\cdot)$. 
\end{proof}

\begin{proposition}\label{lem:lem8}
Fix $\bar x\in E$, $\eta>1$, and $m\in\bN$. Assume $g_m\in C(E)$ admits bounded MPE solution $(w_m,\lambda_m)$ and $\bP$ has unique invariant measure $\nu$. If $\lambda_m=1$, then $\nu(B(\bar{x}, \eta-\frac{1}{m}))=0$. 
\end{proposition}
\begin{proof}
Iterating MPE equation with $\lambda_m=1$, for any $x\in E$ and $n\in\bN$, we have 
\begin{equation}
 w(x)=\mu_x\left(\sum_{i=0}^{n-1}(g_m(x_i)-1)+w(x_n)\right).
\end{equation}
Consequently, since $g_m(\cdot)-1\leq \1_{B(\bar{x}, \eta-\frac{1}{m})^c}(\cdot)-1=-\1_{B(\bar{x}, \eta-\frac{1}{m})}(\cdot)$, for any $n\in\bN$, we have
\begin{equation}
w(x)-\|w\|\leq \mu_x\left(-N(B(\bar{x}, \eta-\tfrac{1}{m}),n)\right),   
\end{equation}
where $N(A,n)$ denotes the number of visits in set $A\in\cE$ in the time interval $[0,n]$. Letting $n\to \infty $, we get $\mu_x\left(-N(B(\bar{x}, \eta-\frac{1}{m}))\right) >-\infty $, where $N(A)$ denotes the number of visits in $A$ over the whole time interval. In particular, this implies that for any $x\in E$ we get
\[
\bP_x\left[N(B(\bar{x}, \eta-\tfrac{1}{m})) < \infty\right]>0.
\]
This property cannot be satisfied if $\nu(B(\bar{x}, \eta-\frac{1}{m}))>0$, which concludes the proof. 
\end{proof}

\begin{corollary}
Under assumptions of Proposition~\ref{lem:lem8}, if $\nu(B(\bar{x}, \eta-\tfrac{1}{m}))>0$, then $\lambda_m<1$, and for sufficiently large $n\in\bN$ we must have 
\[
\sup_{x\in E} \bP_x\left[\tau_B>n\right]<1\quad\textrm{or}\quad \inf_{x\in E}   \bP_x\left[\tau_N\leq n \right]>0.
\]
\end{corollary}
\begin{proof}
Property $\nu(B(\bar{x}, \eta-\frac{1}{m}))>0$ combined with Proposition \ref{lem:lem8} implies $\lambda_m<1$ and it remains to use \eqref{eq:ne2'}.
\end{proof}

Let us now provide a summary of the technical results presented in this section that apply to the general locally compact separable metric space $E$. The study of local bounded MPE solution existence for reward functions that approximate (scaled) set indicator functions proved to be a useful analytical tool. In particular, we conclude that if the underlying Feller Markov process has a unique invariant measure, then it is quite natural for the MPE solution to be unbounded as we expect $\lambda_m=1$ for some non-transient set and sufficiently big $m\in\bN$. Indeed, if the MPE solution is bounded and we have $\lambda_m=1$, then $\nu(B(\bar{x},\eta-{1\over m}))=0$ which means that the corresponding ball must be transient -- this points out to relatively compact-like process propagation.

\section{Negative examples}

Of course, fast process entry into invariant measure support is not sufficient to guarantee existence of MPE under \eqref{A.1}. In particular, even if the invariant measure has full support, the solution to MPE might still not exist. In this section we show a series of negative examples that state conditions under which MPE will not always exist. We hope that the proof techniques introduced in section could be used in the future to help formulate natural necessary and sufficient conditions for MPE existence on general spaces.

\begin{example}[No bounded MPE solution under full support invariant measure]\label{ex:denumerable}
Let $E=\bN$ be a denumerable state space and let us consider a Markov chain with transition matrix
\[
P:=\begin{bmatrix}
\tfrac{1}{2}  & \tfrac{1}{2}& 0 &0 &0 & \ldots\\
\tfrac{1}{2}& 0  &\tfrac{1}{2} &0 &0 & \ldots\\
\tfrac{1}{2}& 0  &0 &\tfrac{1}{2} &0& \ldots\\
\tfrac{1}{2}& 0  &0 &0 &\tfrac{1}{2}& \ldots\\
\ldots & 0&0 &0 &0 & \ldots
\end{bmatrix}.
\]
Then, there exists $g\in C(E)$ for which bounded MPE solution does not exist.
%it would be a nice example in which \eqref{A.3} does not hold but MPE exists.
\end{example}
\begin{proof}
Note that \eqref{A.1} is satisfied with $\Lambda=\frac{1}{2}$.  Let $\epsilon>0$ and let $g\in C(E)$ be given by
\[
g(i):=
\begin{cases}
0,&\textrm{if } i\in\bigcup_{k\in\bN}[2^{k},2^{k}+2^{k-1})\\
2(\ln 2 + \epsilon),&\textrm{if } i\in\bigcup_{k\in\bN}[2^{k}+2^{k-1},2^{k+1})
\end{cases}.
\]
Let us assume that bounded MPE solution, say $\lambda\in\bR$ and $w\in C(E)$, exists; without loss of generality we assume that $w(1)=0$. Let $(x_t)$ denote the Markov process starting at $x_0=1$. Noting that for any $n\in\bN$ we get
\[
\sum_{t=0}^{n}g(x_t) \leq \frac{n\cdot 2(\ln 2+\epsilon)}{2},
\]
we conclude that $\lambda \leq \ln 2 +\epsilon$. Moreover, recalling MPE equation \eqref{eq:3:bellman.av}, we know that for any $i\in\bN$ we get
\begin{align*}
w(i)-w(i+1) &=g(i)-\lambda +\ln\left[\tfrac{1}{2}e^{0}+\tfrac{1}{2}e^{w(i+1)}\right]-w(i+1)\\
&\geq g(i)-(\ln 2+\epsilon)-\ln 2 +\ln\left[1+e^{w(i+1)}\right]-w(i+1)\\
& \geq g(i)-(\epsilon+2\ln 2).
\end{align*}
In particular, for any $k\in\bN$, taking the sum, we get
\begin{align*}
w(2^{k}+2^{k-1})-w(2^{k+1}) &\geq \left[\sum_{i=2^{k}+2^{k-1}}^{2^{k+1}-1}g(i)\right]-(2^{k+1}-2^{k}-2^{k-1})(\epsilon+2\ln 2)\\
& = (2^{k+1}-2^{k}-2^{k-1})(2(\ln2+\epsilon)-\epsilon-2\ln 2)\\
&\geq 2^{k-1}\epsilon.
\end{align*}
This contradicts the assumption that $w\in C(E)$. In fact, one could also show non-existence by investigating iterated values of $\lambda$. Indeed, for any $n\in\bN$ we get
\begin{align*}
\lambda &\leq\tfrac{1}{n}\mu_{1}\left(\sum_{i=0}^{n-1}g_k(x_i)\right)+\tfrac{1}{n}\|w\|_{\textrm{sp}}\\
&\leq \epsilon+\ln 2 +\tfrac{1}{n}\|w\|_{\textrm{sp}},\\
\lambda &\geq\tfrac{1}{n}\mu_{2^{n}+2^{n-1}}\left(\sum_{i=0}^{n-1}g_k(x_i)\right)-\tfrac{1}{n}\|w\|_{\textrm{sp}}\\
&\geq \tfrac{1}{n}\ln\left[\tfrac{1}{2^{n-1}}e^{n2(\ln 2 +\epsilon)}\right] -\tfrac{1}{n}\|w\|_{\textrm{sp}}\\
&\geq \tfrac{1}{n}(-(n-1)\ln2+n2(\ln 2 +\epsilon))-\tfrac{1}{n}\|w\|_{\textrm{sp}}\\
& =2\epsilon+\tfrac{n+1}{n}\ln 2 -\tfrac{1}{n}\|w\|_{\textrm{sp}}.
\end{align*}
Taking the limit, this leads to contradiction, as we get $2\epsilon \leq \lambda-\ln2 \leq \epsilon$.
\end{proof}

Now, let us build upon Example~\ref{ex:denumerable} and present a modification of this example in which we permit one step state recurrence. In the next example, we show how to prove non-existence of MPE bounded solution by introducing two reward functions and use stochastic dominance to analyse relation between them.

\begin{example}[No bounded MPE solution under full support invariant measure and one step recurrent states]\label{ex:denumerable2}
Let $E=\bN_{+}$ be a denumerable state space and let us consider a Markov chain with transition matrix
\[
P:=\begin{bmatrix}
\tfrac{3}{4}  & \tfrac{1}{4}& 0 &0 &0 & \ldots\\
\tfrac{1}{2}& \tfrac{1}{4}  &\tfrac{1}{4} &0 &0 & \ldots\\
\tfrac{1}{2}& 0  &\tfrac{1}{4} &\tfrac{1}{4} &0& \ldots\\
\tfrac{1}{2}& 0  &0 &\tfrac{1}{4} &\tfrac{1}{4}& \ldots\\
\ldots & \ldots& \ldots &\ldots &\ldots & \ldots
\end{bmatrix}.
\]
Then, there exists $g\in C(E)$ for which bounded MPE solution does not exist.
\end{example}

\begin{proof}
Note that \eqref{A.1} is satisfied with $\Lambda=\frac{1}{2}$. Let us assume the MPE solution for $P$ exists for any $g\in C(E)$. For a fixed $k\in\bR_{+}$, we define two functions $g_1,g_2\in B(E)$ by setting
\[
g_1(i):=
\begin{cases}
0,&\textrm{if } i\in 2\bN+1\\
k, &\textrm{if } i\in 2\bN,
\end{cases},
\quad g_2(i):=
\begin{cases}
0,&\textrm{if } i\in\bigcup_{n\in\bN}[2^{n},2^{n}+2^{n-1})\\
k,&\textrm{if } i\in\bigcup_{n\in\bN}[2^{n}+2^{n-1},2^{n+1})
\end{cases}.
\]
Let $\lambda_1,\lambda_2\in\bR$ and $w_1,w_2\in C(E)$ be solutions to MPE equations for $g_1$ and $g_2$, respectively. Without loss of generality, let us assume that $w_1(1)=w_2(1)=0$. Let $S^1_n:=\sum_{t=1}^{n-1}g_1(x_t)$ and $S^2_n:=\sum_{t=1}^{n-1}g_2(x_t)$, where $(x_t)$ denotes the underlying Markov process with starting point $x_0\in E$. Recalling that, for $x_0=1$, we get
\[
\lambda_1=\lim_{n\to\infty}\tfrac{1}{n}\mu_1\left(S^1_n\right)\quad\textrm{and}\quad \lambda_2=\lim_{n\to\infty}\tfrac{1}{n}\mu_1\left(S^2_n\right),
\]
and noting that for $n\in\bN$ we get $F_{S_n^1}(\cdot)\geq F_{S_n^2}(\cdot)$, i.e. $S_n^1$ stochastically dominates $S_n^2$, for $x_0=1$, we conclude that
\begin{equation}\label{eq:g2.lambda.com}
\lambda_1\geq \lambda_2.
\end{equation}
Moreover, by analysing the structure of matrix $P$, it is easy to check that for $i\in\bN\setminus \{1\}$ we get $w_1(i)=w_1(i+2)$. This allows us to explicitly calculate the value of $\lambda_1$. Indeed, for any even $i>1$, from MPE equation \eqref{eq:3:bellman.av}, we get
\begin{align*}
w_1(i) & =g_1(i)-\lambda_1+\ln\left(\frac{1}{2}e^0 +\frac{1}{4}e^{w_1(i)}+\frac{1}{4}e^{w_1(i+1)}\right),\\
& =0-\lambda_1+\ln\left(\frac{1}{2}e^0 +\frac{1}{4}e^{w_1(i)}+\frac{1}{4}e^{w_1(i+1)}\right)\\
w_1(i+1) & =g_1(i+1)-\lambda_1+\ln\left(\frac{1}{2}e^0 +\frac{1}{4}e^{w_1(i+1)}+\frac{1}{4}e^{w_1(i+2)}\right)\\
& =k-\lambda_1+\ln\left(\frac{1}{2}e^0 +\frac{1}{4}e^{w_1(i+1)}+\frac{1}{4}e^{w_1(i)}\right),
\end{align*}
which implies $w_1(i)-w_1(i+1)=-k$. Consequently, for $i=2$, we get
\begin{align*}
w_1(2) & =-\lambda_1+\ln\left(\frac{1}{2}e^0 +\frac{1}{4}e^{w_1(2)}+\frac{1}{4}e^{w_1(2)+k}\right),
\end{align*}
which implies $\lambda_1=\ln\left(\frac{1}{2}e^{-w_1(2)} +\frac{1}{4}e^{0}+\frac{1}{4}e^{k}\right)$.
On the other hand, by MPE equation \eqref{eq:3:bellman.av} applied to $i=1$, we also know that $\lambda_1=\ln\left(\frac{3}{4}e^0 +\frac{1}{4}e^{w_1(2)}\right)$, which implies
\[
\tfrac{3}{4}e^0 +\tfrac{1}{4}e^{w_1(2)}=\tfrac{1}{2}e^{-w_1(2)} +\tfrac{1}{4}e^{0}+\tfrac{1}{4}e^{k}
\]
%and consequently
%\[
%2e^{-w_1(2)}-e^{w_1(2)}=2-e^{k}.
%\]
Thus, simple algebraic calculations yield
\[
e^{w_1(2)}=\tfrac{1}{2}\left(\sqrt{12+e^{2k}-4e^{k}}+e^k-2\right)
\]
and consequently we get
\begin{align}
\lambda_1& =\ln\left(\frac{1}{2} +\frac{1}{8}\left(\sqrt{12+e^{2k}-4e^{k}}+e^k\right)\right).\label{eq:g1.lambda1}
%& =-\ln 2 +\ln\left(1 +\frac{1}{4}\left(\sqrt{12+e^{2k}-4e^{k}}+e^k\right)\right)
\end{align}
Now, let us show that solution for $g_2$ does not exists. From \eqref{eq:g2.lambda.com} and \eqref{eq:g1.lambda1} we know that for sufficiently big $k>0$ (used to define $g_1$ and $g_2$), and the corresponding MPE solutions, we get
\begin{equation}\label{eq:g2.1}
\lambda_2+\ln 2-k\leq \lambda_1+\ln 2-k < \ln\left(e^{-k} +\frac{1}{4}\left(e^{-k}\sqrt{12+e^{2k}}+1\right)\right)<0.
\end{equation}
Recalling that we set $w_2(1)=0$ and using MPE equation for $i\geq 1$, we get
% \begin{align*}
% w_2(i) & =g_2(i)-\lambda_2+\ln\left(\frac{1}{2}e^0 +\frac{1}{4}e^{w_2(i)}+\frac{1}{4}e^{w_2(i+1)}\right)\\
% & = g_2(i)-\lambda_2-\ln 2 +\ln\left(\frac{1}{2}(1 +e^{w_2(i)})+\frac{1}{2}(1+e^{w_2(i+1)})\right)\\
% & \geq g_2(i)-\lambda_2-\ln 2 +\frac{1}{2}\ln\left(1 +e^{w_2(i)}\right)+\frac{1}{2}\ln\left(1 +e^{w_2(i+1)}\right)\\
% & \geq g_2(i)-\lambda_2-\ln 2 +\frac{1}{2}w_2(i)+\frac{1}{2}w_2(i+1)
% \end{align*}
\begin{align*}
w_2(i) & =g_2(i)-\lambda_2+\ln\left(\frac{1}{2}e^0 +\frac{1}{4}e^{w_2(i)}+\frac{1}{4}e^{w_2(i+1)}\right)\\
& \geq g_2(i)-\lambda_2+\ln\left(\frac{1}{4}e^{w_2(i)}+\frac{1}{4}e^{w_2(i+1)}\right)\\
& \geq g_2(i)-\lambda_2-\ln 2 +\ln\left(\frac{1}{2}e^{w_2(i)}+\frac{1}{2}e^{w_2(i+1)}\right)\\
& \geq g_2(i)-\lambda_2-\ln 2 +\frac{1}{2}w_2(i)+\frac{1}{2}w_2(i+1)
\end{align*}
% {\color{red} Simpler?:
% \begin{align*}
% w_2(i) & =g_2(i)-\lambda_2+\ln\left(\frac{1}{2}e^0 +\frac{1}{4}e^{w_2(i)}+\frac{1}{4}e^{w_2(i+1)}\right)\\
% & \geq g_2(i)-\lambda_2+\ln\left(\frac{1}{4}e^{w_2(i+1)}\right)\\
% & =  g_2(i)-\lambda_2-\ln 4 +w_2(i+1)
% \end{align*}
% }
which implies
\begin{equation}\label{eq:g2.2}
w_2(i)-w_2(i+1)\geq 2(g_2(i)-\lambda_2-\ln 2)\quad \textrm{for } i\geq 1.
\end{equation}
In particular, recalling definition of $g_2$ and using \eqref{eq:g2.2}, for any $n\in\bN$ we get
\[
w_2(2^{n+1})-w_2(2^{n}+2^{n-1}) \geq 2^{n-1}\cdot 2(k-\lambda_2-\ln 2).
\]
Now, recalling the fact that $k-\lambda_2-\ln2 >0$ due to \eqref{eq:g2.1} and setting $n\to\infty$, we get that $\|w_2\|_{\textrm{sp}}=\infty$, which contradicts the assumption that MPE solution for $g_2$ is bounded.
\end{proof}

In the last two example, we assumed that there are states which cannot be connected (with positive probability) in a finite number of steps and that there is a single path that goes to infinity and we can stay on it with geometric probability. In the next example, we show that those conditions could be relaxed, i.e. we can allow two-step connection (with positive probability) and split the geometric propagation into different paths.

\setcounter{MaxMatrixCols}{20}
\begin{example}[No bounded MPE solution under full support invariant measure with disjoint geometric propagation]
Let $E=\bN_{+}$ be a denumerable state space and let us consider a Markov chain with transition matrix
\[
P:=\begin{bmatrix}
1/2+p &\frac{e^{-2}}{2} &\frac{e^{-2}}{2} & \frac{e^{-4}}{3} & \frac{e^{-4}}{3} & \frac{e^{-4}}{3}& \frac{e^{-8}}{4}&  \frac{e^{-8}}{4}& \frac{e^{-8}}{4} & \frac{e^{-8}}{4} & \frac{e^{-16}}{5} & \frac{e^{-16}}{5} &\ldots\\
1/2& 0  &1/2    &0      &0       &0        & 0& 0& 0 & 0 & 0 & 0 &\ldots\\
1& 0    &0      &0      &0       &0        & 0& 0& 0 & 0 & 0 & 0 &\ldots\\
1/2& 0  &0      &0      &1/2     &0        & 0& 0& 0 & 0 & 0 & 0 &\ldots\\
1/2& 0  &0      &0      &0       &1/2      & 0& 0& 0 & 0 & 0 & 0 &\ldots\\
1  & 0  &0      &0      &0       &0        & 0& 0& 0 & 0 & 0 & 0 &\ldots\\
1/2& 0  &0      &0      &0       &0        & 0& 1/2& 0 & 0 & 0 & 0 &\ldots\\
1/2& 0  &0      &0      &0       &0        & 0& 0& 1/2 & 0 & 0 & 0 &\ldots\\
1/2& 0  &0      &0      &0       &0        & 0& 0& 0 & 1/2 & 0 & 0 &\ldots\\
1& 0  &0      &0      &0       &0        & 0& 0& 0 & 0 & 0 & 0 &\ldots\\
1/2  & 0  &0      &0      &0       &0        & 0& 0& 0 & 0 & 0 & 1/2 &\ldots\\
1/2  & 0  &0      &0      &0       &0        & 0& 0& 0 & 0 & 0 & 0 &\ldots\\
\ldots  & \ldots  &\ldots      &\ldots      &\ldots       &\ldots        & \ldots& \ldots& \ldots & \ldots & \ldots & \ldots &\ldots
\end{bmatrix}.
\]
where $p:=\frac{1}{2}-\sum_{i=1}^{\infty}e^{-2^i}$. Then, there exists $g\in C(E)$ for which bounded MPE solution does not exist.
\end{example}

\begin{proof}
let us split the state space into disjoint sets $A_0=\{1\}$, $A_1:=\{2,3\}$, $A_2:=\{4,5,6\}$, $A_3:=\{7,8,9,10\}$, etc. Using convention $A_i:=\{a^i_1,\ldots,a^i_{i+1}\}$, $i\in\bN$, let the transition kernel be given by by the following transition probabilities
\[
\bP(a_j^i,1)=
\begin{cases}
\frac{1}{2}, & i\in\bN_{+},\, j=1,\ldots,i\\
1, & i\in\bN_{+},\, j=i+1
\end{cases},\quad
\bP(a^i_j,a^i_{j}+1)=
\begin{cases}
\frac{1}{2} & i\in\bN_{+},\, j=1,\ldots,i\\
0 & i\in\bN_{+},\, j=i+1
\end{cases},
\]
\[
\bP(1,a^i_j)=
\begin{cases}
\frac{1}{2}+p & i=0, j=1\\
\frac{1}{i+1}e^{-2^{i}} & i\neq 0,\, j=1,2,\ldots,i+1
\end{cases},
\]
Note that the stated values fully allocate non-zero probability transitions. Next, let $g\in C(E)$ be given by

\[
g(a^i_j):=k(1-\1_{\{j=i+1\}}),\quad i\in\bN,\, j=1,\ldots,i+1
\]
Note that $g$ could be linked to the reward value vector given by 
\[
(0,k,0,k,k,0,k,k,k,0,k,k,k,k,0,\ldots).
\]
For simplicity, we also use $p_i$ to denote the probability of entering set $A_i$ from state $1$, i.e. we set
\[
p_i:=\bP_{1}[A_i] =\sum_{k=1}^{i+1}e^{-2^i}/(i+1)=e^{-2^i}.
\]
Let us assume that a bounded solution to MPE equation for $g$ exists and reach a contradicition; as usual, we use $\lambda\in\bR$ and $w\in C(E)$ to refer to the solution. First, let us show that
\begin{equation}\label{eq:lambda.lower1}
\lambda\geq  k+\ln\tfrac{1}{2}.
\end{equation}
For any $n\in\bN$, noting that $g$ is non-negative, we have
\[
n\lambda\geq \mu_{a^n_1}\left(\sum_{i=0}^{n-1}g(x_i)\right)-\|w\|\geq \ln\left[\frac{1}{2^n}e^{nk}\right] -\|w\| =nk +n\ln\tfrac{1}{2}-\|w\|.
\]
Dividing both sides by $n$ and letting $n\to\infty$ we get \eqref{eq:lambda.lower1}.
Second, to contradict bounded MPE existence, we want to show that there exists $k\in\bR_{+}$ such that
\begin{equation}\label{eq:lambda.upper1}
\lambda<k+\ln\tfrac{1}{2}.
\end{equation}
For simplicity, we set $k=2$. To show \eqref{eq:lambda.upper1}, let us first define a sequence of (discrete) i.i.d. random variables $(Z_i)_{i\in\bN}$ with probability mass function characterised by
\[
\textstyle \bP[Z_1=0] =1-\sum_{i=1}^{\infty} p_i,\quad
\bP[Z_1=ik] =p_i, \quad i\in\bN.
\]
Then, for starting point $x_0=1$ and any $n\in\bN$ one could show that
\begin{equation}\label{eq:weird1}
\textstyle \sum_{i=0}^{n}g(x_i) \preceq \sum_{i=1}^{n}Z_i,
\end{equation}
where $\preceq$ denotes the first order stochastic dominance. For brevity, we only sketch the idea of the proof of \eqref{eq:weird1}. First, to show \eqref{eq:weird1} we need to look at state $x_0=1$ as the 'reset state'. The dynamics of $V_k=\sum_{i=0}^{k}g(x_i)$, for $k=1,\ldots,n$, is as follows: (1) the process starts in $x_0=1$ with $V_0=0$; (2) in the first time step, the process remains at state $1$ (no increase) or enters subset $A_i$; (3) if the process is in $A_i$ then the process re-enters state $1$ in the (random) number of steps bounded from above  by $i$ -- the total aggregated reward before re-entering state $1$ is bounded by $ik$; (4) when the process re-enters state $1$ the dynamics is reset. Now, since $g$ is non-negative, the value of $V_k$ could be bounded from above. Second, we note that the sequence $(Z_i)$ could be linked to immediate gratification sequence. The construction is as follows: (1) we assume that if process enters $A_i$ (from state $1$) then we are paid an immediate full gratification $ik$ in one step and reset the state instantly; (2) this is better compared to the previous case since we assume we are paid immediately the maximal possible reward in a single step and we profit both from getting maximal reward with probability one and from increasing the number of time the process 'resets' itself after getting the reward; (3) this procedure leads to the sum of i.i.d. random variables $Z_i$ as presented in \eqref{eq:weird1}.

To show why this helps to prove \eqref{eq:lambda.upper1} let us note that the entropy of $Z_1$ could be bounded from above by $2+\ln\tfrac{1}{2}$. Namely, we get
\begin{equation}\label{eq:weuiord2}
\mu(Z_1)\leq \ln\left[ 1 e^0+ \sum_{i=1}^{\infty}e^{-2^i}e^{ik} \right]\leq \ln\left[ 1+ \sum_{i=1}^{\infty}e^{ik-2^i} \right] \leq \ln (3.15) < 2+\ln\tfrac{1}{2}.
\end{equation}
Now, noting that entropic utility is monotone, law-invariant, additive with respect to independent random variables, using \eqref{eq:weird1} and \eqref{eq:weuiord2}, and fixing starting point $x_0=1$, for $n\in\bN$, we get 
\[
\frac{1}{n}\mu_{x_1}\left(\sum_{i=0}^{n-1}g(x_i)\right)\leq \frac{1}{n}\mu\left(\sum_{i=1}^{n}Z_i\right)=\frac{n}{n}\mu(Z_1) < 2+\ln\tfrac{1}{2}.
\]
Thus, letting $n\to\infty$, we get 
\[
\lambda\leq \lim_{n\to\infty}\left[ \frac{1}{n}\mu_{x_1}\left(\sum_{i=0}^{n-1}g(x_i)\right)+\frac{\|w\|}{n}\right] < 2+\ln\tfrac{1}{2},
\]
which proves \eqref{eq:lambda.upper1} and leads to contradiction with bounded MPE existence.
\end{proof}

Finally, let us show, that local geometric propagation, i.e. positive lower bound imposed on diagonal values in the transition matrix, is sufficient to guarantee MPE non-existance for some $g\in C(E)$.

\begin{example}[No bounded MPE solution under full support invariant measure with local geometric propagation]
Let $E=\bN_{+}$ be a denumerable state space and let us consider a Markov chain with transition matrix
\[
P:=\begin{bmatrix}
\tfrac{1}{2}  & a_1& a_2 &a_3 &a_4 & \ldots\\
\tfrac{1}{2}& \tfrac{1}{2}  &0 &0 &0 & \ldots\\
\tfrac{1}{2}& 0  & \tfrac{1}{2} &0 &0& \ldots\\
\tfrac{1}{2}& 0  &0 &\tfrac{1}{2} &0& \ldots\\
\ldots & \ldots& \ldots &\ldots &\ldots & \ldots
\end{bmatrix},
\]
where $(a_i)_{i=1}^{\infty}$ is such that $a_i>0$, $i\in\bN$, and $\sum_{i=1}^{\infty}a_i=\tfrac{1}{2}$.  Then, there exists $g\in C(E)$ for which bounded MPE solution does not exist.
\end{example}

\begin{proof}
 The proof follow directly from the results presented in Section 9 in \cite{Cav2018} and is omitted for brevity. To see the link, observe that there exists a decreasing subsequence $(a_{i_k})_{k\in\bN}$ for which we have $a_{i_k}< \zeta(3)/(1+k)^3$, where $\zeta(3)$ is the Ap\'ery's constant and set
\[
g(x)=
\begin{cases}
2\left(\ln 2+\ln\left(1-\frac{1}{k+2} \right)\right) & \textrm{if } x=a_{i_k} \textrm{ and } k\geq N,\\
0 & \textrm{otherwise},
\end{cases}
\]
for some large constant $N\in\bN$; see Proposition 9.1 and Proposition 9.2 in \cite{Cav2018} for details. Then, one can show that constant $\lambda$ in MPE equation must depend on the state, i.e. for state 1 it must be smaller than for sufficiently large state, which would lead to contradiction, see also Proposition 9.3 and Proposition 9.4 in \cite{Cav2018}.
\end{proof}

From the examples introduced in this section one can deduce that the Markov process should propagate in a rather compact and uniform manner. The core idea behind all denumerable examples introduced in this section was to construct a process for which: (1) when starting from state $1$, it is hard to reach states associated with big integers; (2) when starting far away from $1$, it is possible to remain far away from $1$ with some geometric probability. Hopefully, this intuition could be expanded to provide in a future a full MPE existence characterisation results, with conditions possibly linked to geometric propagation as the one presented in Proposition~\ref{pr:escape.geometric}. This question remain open for the future research.

\section*{Acknowledgements}
The authors would like to express their sincere thanks to Jan Palczewski, currently at University of Leeds, for invaluable comments on the draft of this manuscript. In particular, formulation and proof concept of Proposition~\ref{pr:escape.geometric} followed from discussions with Prof. Palczewski.

\bibliographystyle{siamplain}
\bibliography{PiteraStettner2023}

\begin{thebibliography}{10}

\bibitem{BalMey2000}
{\sc S.~Balaji and S.~P. Meyn}, {\em Multiplicative ergodicity and large
  deviations for an irreducible {M}arkov chain}, Stochastic processes and their
  applications, 90 (2000), pp.~123--144.

\bibitem{BieCiaPit2013}
{\sc T.~R. Bielecki, I.~Cialenco, and M.~Pitera}, {\em Dynamic limit growth
  indices in discrete time}, Stochastic Models, 31 (2015), pp.~494--523.

\bibitem{BiePli1999}
{\sc T.~R. Bielecki and S.~R. Pliska}, {\em Risk-sensitive dynamic asset
  management}, Applied Mathematics \& Optimization, 39 (1999), pp.~337--360.

\bibitem{BiePli2003}
{\sc T.~R. Bielecki and S.~R. Pliska}, {\em Economic properties of the risk
  sensitive criterion for portfolio management}, Review of Accounting and
  Finance, 2 (2003), pp.~3--17.

\bibitem{BisBor2022}
{\sc A.~Biswas and V.~S. Borkar}, {\em Ergodic risk-sensitive control -- a
  survey}, arXiv preprint arXiv:2301.00224,  (2022).

\bibitem{BisPra2022}
{\sc A.~Biswas and S.~Pradhan}, {\em Ergodic risk-sensitive control of {M}arkov
  processes on countable state space revisited}, ESAIM: COCV, 28 (2022), p.~26.

\bibitem{Cav2009}
{\sc R.~Cavazos-Cadena}, {\em The risk-sensitive {P}oisson equation for a
  communicating {M}arkov chain on a denumerable state space}, Kybernetika, 45
  (2009), pp.~716--736.

\bibitem{Cav2018}
{\sc R.~Cavazos-Cadena}, {\em Characterization of the optimal risk-sensitive
  average cost in denumerable {M}arkov decision chains}, Mathematics of
  Operations Research, 43 (2018), pp.~1025--1050.

\bibitem{CavHer2009}
{\sc R.~Cavazos-Cadena and D.~Hern{\'a}ndez-Hern{\'a}ndez}, {\em Necessary and
  sufficient conditions for a solution to the risk-sensitive {P}oisson equation
  on a finite state space}, Systems \& control letters, 58 (2009),
  pp.~254--258.

\bibitem{CavHer2010}
{\sc R.~Cavazos-Cadena and D.~Hern{\'a}ndez-Hern{\'a}ndez}, {\em Poisson
  equations associated with a homogeneous and monotone function: necessary and
  sufficient conditions for a solution in a weakly convex case}, Nonlinear
  Analysis: Theory, Methods \& Applications, 72 (2010), pp.~3303--3313.

\bibitem{DiMSte1999}
{\sc G.~B. Di~Masi and {\L}.~Stettner}, {\em Risk-sensitive control of
  discrete-time {M}arkov processes with infinite horizon}, SIAM Journal on
  Control and Optimization, 38 (1999), pp.~61--78.

\bibitem{DiMSte2000}
{\sc G.~B. Di~Masi and {\L}.~Stettner}, {\em Infinite horizon risk sensitive
  control of discrete time {M}arkov processes with small risk}, Systems \&
  control letters, 40 (2000), pp.~15--20.

\bibitem{DiMSte2006}
{\sc G.~B. Di~Masi and {\L}.~Stettner}, {\em On additive and multiplicative
  (controlled) {P}oisson equations}, Banach Center Publications, 72 (2006),
  p.~57.

\bibitem{DiMSte2007}
{\sc G.~B. Di~Masi and {\L}.~Stettner}, {\em Infinite horizon risk sensitive
  control of discrete time {M}arkov processes under minorization property},
  SIAM Journal on Control and Optimization, 46 (2007), pp.~231--252.

\bibitem{DonskerVI}
{\sc M.~Donsker and S.~Varadhan}, {\em Asymptotic evaluation of certain
  {M}arkov process expectations for large time, {I}}, Comm. Pure Appl. Math.,
  28 (1975), pp.~1--47.

\bibitem{DonskerVIII}
{\sc M.~Donsker and S.~Varadhan}, {\em Asymptotic evaluation of certain
  {M}arkov process expectations for large time, {III}}, Comm. Pure Appl. Math.,
  29 (1976), pp.~389--461.

\bibitem{Doob}
{\sc J.~Doob}, {\em Stochastic Processes}, Wiley, 1990.

\bibitem{DuncanSt2000}
{\sc T.~Duncan, B.~Pasik-Duncan, and {\L}.~Stettner}, {\em Adaptive control of
  discrete time {M}arkov processes by the large deviations method},
  Applicationes Mathematicae, 27 (2000), pp.~265--285.

\bibitem{HerLas1996}
{\sc O.~Hern{\'a}ndez-Lerma and J.~B. Lasserre}, {\em Discrete-time {M}arkov
  control processes}, Springer, 1996.

\bibitem{HerLas1999}
{\sc O.~Hern{\'a}ndez-Lerma and J.~B. Lasserre}, {\em Further topics on
  discrete-time Markov control processes}, vol.~42, Springer Science \&
  Business Media, 1999.

\bibitem{JelPitSte2020}
{\sc D.~Jelito, M.~Pitera, and {\L}.~Stettner}, {\em Long-run risk-sensitive
  impulse control}, SIAM Journal on Control and Optimization, 58 (2020),
  pp.~2446--2468.

\bibitem{JelPitSte2021}
{\sc D.~Jelito, M.~Pitera, and {\L}.~Stettner}, {\em Risk sensitive optimal
  stopping}, Stochastic Processes and their Applications, 136 (2021),
  pp.~125--144.

\bibitem{KonMey2003}
{\sc I.~Kontoyiannis and S.~P. Meyn}, {\em {S}pectral theory and limit theorems
  for geometrically ergodic {M}arkov processes}, Annals of Applied Probability,
   (2003), pp.~304--362.

\bibitem{MenRob2018}
{\sc J.~Menaldi and M.~Robin}, {\em On some ergodic impulse control problems
  with constraint}, SIAM Journal on Control and Optimization, 56 (2018),
  pp.~2690--2711.

\bibitem{MeynTweedie}
{\sc S.~Meyn and R.~Tweedie}, {\em Markov Chains and Stochastic Stability},
  Springer, 1996.

\bibitem{PitSte2016}
{\sc M.~Pitera and L.~Stettner}, {\em Long run risk sensitive portfolio with
  general factors}, Mathematical Methods of Operations Research, 83 (2016),
  pp.~265--293.

\bibitem{PitSte2019}
{\sc M.~Pitera and {\L}.~Stettner}, {\em Long-run risk sensitive dyadic impulse
  control}, Applied Mathematics \& Optimization, 84 (2021), pp.~19--47.

\bibitem{Ste1999}
{\sc {\L}.~Stettner}, {\em Risk sensitive portfolio optimization}, Mathematical
  Methods of Operations Research, 50 (1999), pp.~463--474.

\bibitem{ste2022}
{\sc {\L}.~Stettner}, {\em On an approximation of average cost per unit time
  impulse control of {M}arkov processes}, SIAM Journal on Control and
  Optimization, 60 (2022), pp.~2115--2131.

\end{thebibliography}
\end{document}